\documentclass[a4paper,12pt,twoside]{article}
\usepackage[latin1]{inputenc} 
\usepackage{graphicx}
\usepackage[pdftex,bookmarks]{hyperref}
\usepackage{amsmath}    
\usepackage{amssymb}
\usepackage{amsmath,amsthm}
\usepackage{textcomp}
\usepackage{color}
\usepackage[all]{xy}
\usepackage{slashed}
\usepackage{latexsym}
\usepackage{tikz}
\usepackage{mathrsfs}

\usepackage[titletoc,toc,title]{appendix}

 \usepackage{color}

\usepackage{fancyhdr}
\usepackage{mathpazo}
\usepackage[scaled=.95]{helvet}
\usepackage{courier}

\numberwithin{equation}{section}
\newtheorem{theorem}{Theorem}[section]
\newtheorem{lemma}[theorem]{Lemma}
\newtheorem{proposition}[theorem]{Proposition}

\theoremstyle{definition}
\newtheorem{definition}[theorem]{Definition}
\newtheorem{example}[theorem]{Exemple}
\newtheorem{examples}[theorem]{Exemples}
\newtheorem{remark}[theorem]{Remark}
\theoremstyle{plain}

\def\BB{\ensuremath\mathbb{B}}

\def\CC{\ensuremath{\mathbb{C}}}

\def\RR{\ensuremath{\mathbb{R}}}

\date{}
\begin{document}
\pagestyle{plain}
\title{Interior Kasparov products for $\varrho$-classes on Riemannian foliated bundles}
	\author
	{Vito Felice Zenobi}
	\maketitle

	\begin{abstract}
		Let $\iota\colon \mathcal{F}_0\to\mathcal{F}_1 $ be a suitably oriented inclusion of foliations over a manifold $M$, then we extend the construction of the lower shriek maps given by Hilsum and Skandalis to adiabatic deformation groupoid C*-algebras: we construct an asymptotic morphism $(\iota_{ad}^{[0,1)})_!\in E_n\left(C^*(G_{ad}^{[0,1)}), C^*(H_{ad}^{[0,1)})\right)$, where  $G$ and $H$ are the monodromy groupoids associated with $\mathcal{F}_0$ and $\mathcal{F}_1$ respectively. Furthermore, we prove an interior Kasparov product formula for foliated $\varrho$-classes associated with longitudinal metrics of positive scalar curvature in the case of Riemannian foliated bundles. 
	\end{abstract}
	
	\medskip
	
	\tableofcontents 
\section{Introduction}

Let $M$ be a smooth compact manifold with fundamental group $\pi_1(M)$ and let $\widetilde{M}$ be its universal covering.
 In \cite{Higson-Roe} the authors introduced the following long exact sequence of K-theory groups
 \begin{equation}\label{hr-seq}
 \xymatrix{\cdots\ar[r]& K_*(C^*(\pi_1(M)))\ar[r]& \mathrm{S}^{\pi_1(M)}_{*+1}(\widetilde{M})\ar[r]& K^{\pi_1(M)}_{*+1}(\widetilde{M})\ar[r]&\cdots}
\end{equation}
known as the Higson-Roe exact sequence or the \emph{analytic surgery exact sequence}. The group on the right is the equivariant K-homology of $\widetilde{M}$; on the left we have the K-theory of the group C*-algebra, the receptacle of the index invariants; finally, the middle term is the so-called \emph{structure group} whom the K-theoretic secondary invariants belong to.

These invariants, usually called $\varrho$-classes, were introduced in  \cite{Higson-Roe}. They are used to detect secondary geometric or topological structures, such as topological/differential manifold structures or metrics with positive scalar curvature, see for instance \cite{PS, PS2, XY,XY2, WXY, Z}. All these papers treat the subject within the methodologies of coarse geometry.

In \cite{zenobi-ad}, K-theoretic secondary invariants were adapted to the context of Lie groupoids. 
Indeed, let $G\rightrightarrows M$ be a Lie groupoid over $M$, $AG$ its Lie algebroid and $G_{ad}^{[0,1)}$ its adiabatic deformation, whose precise definition is not important for the moment. Then we can consider the following long exact sequence of K-groups 
\begin{equation}\label{ad-seq}
\xymatrix{\cdots\ar[r]&K_*(C^*(G)\otimes C_0(0,1))\ar[r]&K_*(C^*(G^{[0,1)}_{ad}))\ar[r]&K_*(C^*(AG))\ar[r]&\cdots}.
\end{equation}
In \cite{Z-CvsAd} it is proved that, when $G$  is the fundamental groupoid of $M$, the long exact sequence \eqref{ad-seq}  is isomorphic to the Higson-Roe exact sequence \eqref{hr-seq}
 and that the $\varrho$-classes defineded in \cite{zenobi-ad} correspond to those defined in \cite{PS}. 

 This was a natural setting for a systematized and conceptual generalization to foliations - both regular and singular ones, see also \cite{Piazza-Zenobi} - of the $\varrho$-invariants in K-theory. 
It is worth to mention that, if for long time the coarse geometric methods were not fit to treat foliations, in the recent works \cite{BR1,BR2} Benameur and Roy contributed to advance in this direction.

The specific motivation of this work is to improve the study of Kasparov products for secondary classes. This topic, in the general setting of Lie groupoids, was inspired by a seminal work of Siegel \cite{Sieg} where an exterior product involving secondary classes was defined and then independently generalized to the equivariant setting in \cite{Z} and \cite{Zeidler}. It is given by a map as follows
\begin{equation}
\boxtimes\colon\mathrm{S}^{\pi_1(M)}_{i}(\widetilde{M})\times K_{j}(N)\to \mathrm{S}^{\pi_1(M)\times\pi_1(N)}_{i+j}(\widetilde{M}\times\widetilde{N})
\end{equation}
and it gives a factorization of a $\varrho$-class $\varrho(M\times N)$  as the product of the analogous $\varrho$-class $\varrho(M)$  on $M$ and the fundamental class $[N]$. 
In \cite[Section 3.6]{zenobi-ad} a similar product was introduced in  the more general context of Lie groupoid
\begin{equation}
\boxtimes\colon K_i(C^*(G^{[0,1)}_{ad}))\times K_j(C^*(AH))\to K_{i+j}(C^*((G\times H)^{[0,1)}_{ad}))
\end{equation}
and similar product formulas are proved too. 
The main geometrical motivation of this construction is to obtain stability results about secondary structures, such as rigidity results for topological structures or for metrics with positive scalar curvature.

In the coarse geometric setting the topic  was deepened in other works such as \cite{EWZ,W}.  
An interesting advancement appeared in the recent paper \cite{Hongzhi-Jinmin}, where the authors consider the case of a manifold $M$ fibered over another manifold $B$: they define an interior product between the longitudinal secondary class  on $M$ and the K-homology class of $B$ and prove a product formula for the total secondary invariant of $M$. 

In the present article we cope the task of putting the ad hoc construction given in \cite{Hongzhi-Jinmin} into the more systematic setup offered by \cite{HSk}. In this classical paper, Hilsum and Skandalis provide the  construction of a KK-element associated with a K-oriented map of foliated manifolds and they proved the functoriality (namely what, in this text, we are calling interior product formulas) of their construction for the index classes associated with longitudinal operators. 

Let us give a quick overview of what is developed in the present paper. First, let us fix a manifold $M$ and assume that there is a regular foliation associated with an involutive sub-bundle $\mathcal{F}\subset TM$ over it. We can then consider the monodromy groupoid $G:=\mathrm{Mon}(M,\mathcal{F})\rightrightarrows M$ associated with this foliation. As before,  we can associate with $G$ a long exact sequence in K-theory as in \ref{ad-seq}.
Consider now for instance the inclusion of $AG$, the Lie algebroid of $G$, into $TM$: this inclusion is an homomorphism of Lie algebroids which integrates to an immersion of Lie groupoids $\iota\colon G\looparrowright H$. Here $H$ denotes the fundamental groupoid $\widetilde{M}\times_{\pi_1(M)}\widetilde{M}\rightrightarrows M$. 

The focus of this paper is to construct a mapping of the exact sequence \eqref{ad-seq} associated with $G$ to the one associated with $H$, isomorphic to the classical Higson-Roe exact sequence, as follows
\begin{equation}\label{mapping}
\xymatrix{\cdots\ar[r]&K_*(C^*(G)\otimes C_0(0,1))\ar[r]\ar[d]^{\iota_!}&K_*(C^*(G^{[0,1)}_{ad}))\ar[r]\ar[d]^{(\iota_{ad}^{[0,1)})_!}&K_*(C^*(AG))\ar[r]\ar[d]^{d\iota_!}&\cdots\\
	\cdots\ar[r]&K_*(C^*(H)\otimes C_0(0,1))\ar[r]&K_*(C^*(H^{[0,1)}_{ad}))\ar[r]&K_*(C^*(AH))\ar[r]&\cdots}.
\end{equation}
where the vertical maps are given by asymptotic morphisms. 
 This was already done by Hilsum and Skandalis in \cite{HSk} for the K-groups on the left and the right hand side of \eqref{mapping}. Here we use a mapping cone construction to generalize their construction to the middle term of \eqref{ad-seq}. Moreover, whereas the lower shriek maps of Hilsum and Skandalis where defined only for ameanable groupoids by means of KK-elements, our maps are defined for general groupoids: this is only due to the fact that we can implement this mapping through asymptotic morphisms in E-theory whereas, at the time when \cite{HSk} was written, E-theory did not exist yet. 

The definition of these asymptotic morphisms is only a partial aim of this paper. Indeed, the main result is given by Theorem \ref{main-th}: take a manifold $M$ which is the total space of a Riemannian foliated bundle associated with $\mathcal{F}\subset TM$. Assume that there is a longitudinal metric $g^{\mathcal{F}}$ with positive scalar curvature and let $\varrho(g^{\mathcal{F}})\in K_*(C^*(\mathrm{Mon}(M,\mathcal{F})_{ad}^{[0,1)}))$ be the corresponding $\varrho$-class. Let us then consider the immersion of Lie groupoids $\iota\colon\mathrm{Mon}(M,\mathcal{F})\looparrowright \widetilde{M}\times_{\pi_1(M)}\widetilde{M}$ obtained by integrating the inclusion of Lie algebroids $\mathcal{F}\hookrightarrow TM$. Assume that  the normal bundle admits a $Spin$ structure  and let  \begin{equation}\label{!}(\iota_{ad}^{[0,1)})_!\in E_n\left(C^*(\mathrm{Mon}(M,\mathcal{F}))_{ad}^{[0,1)}),C^*(\widetilde{M}\times_{\pi_1(M)}\widetilde{M}_{ad}^{[0,1)})\right)\end{equation} be the asymptotic morphism associated with it. Let $g^M:= g^{\mathcal{F}}\oplus g^{\mathcal{N}}$ be a suitable metric of $TM$ with positive scalar curvature. Then we can prove the following factorization formula  for the $\varrho$-class of $g^M$
\begin{equation}\label{pf}
\varrho(g^M)=\varrho(g^\mathcal{F})\otimes (\iota_{ad}^{[0,1)})_! \in K_{*+n}\left(C^*(\widetilde{M}\times_{\pi_1(M)}\widetilde{M}_{ad}^{[0,1)})\right).
\end{equation}  

While the construction of \eqref{!} can be given for the most general case of suitably oriented immersion of foliations, proving the product formula \eqref{pf} presents more difficulties. The main issue is that we are dealing  with interior Kasparov product between elements in KK-groups of mapping cone C*-algebras. In particular, we need to prove that the Kasparov product of a degenerate Kasparov bimodule on the foliation and a transverse non-degenerate Kasparov bimodule corresponds to a precise degenerate Kasparov bimodule on the whole manifold. The key technical result to solve this problem is \cite[Lemma 11]{Sk-remarks}, which allows to construct an operatorial homotopy through invertible operators - namely through degenerate bimodules - between the longitudinal operator and the total one, whenever one has a fine estimation from below of their commutator. Fundamental to this end is the calculation of Bismut and Cheeger in \cite[Equation 4.46]{BismutCheeger}, see also \cite[Lemma 17]{KvS2}, for the simple foliation given by a surjective submersion $\pi\colon M\to B$. 
Transposing all that to the equivariant context and passing to quotients leads to prove formula  \eqref{pf} in the case of a Riemannian foliated bundle. The general case stays open. 

As a by-product of the construction \eqref{!} we also obtain a generalization to the context of general foliations of the main theorem in \cite{GXY} about the functoriality of $\varrho$-classes with respect to maps of principal bundle over the same manifold $M$. 

\bigskip

The paper is organized as follows: in Section \ref{s1} we introduce the elementary definitions and notations  about Lie groupoids and their C*-algebras; in Section \ref{s2} the elementary KK-theoretic tools are recalled: the KK-theory of mapping cone C*-algebras, the definition of the bimodules implementing the Bott periodicity, the Thom isomorphisms in KK-theory defined in \cite{HSk}, the asymptotic morphism associated with a deformation groupoid; in Section \ref{longitudinal-classes} we recall the definitions of the primary invariants, the secondary invariants and the fundamental classes as elements of the first, the second and the third K-theory group in \eqref{ad-seq} respectively;
Section \ref{s4} is dedicated to the construction of the elements \eqref{!} both in the case of an isometric and an almost isometric immersion of foliations; Section \ref{s5} is devoted to set the previous constructions in the equivariant setting and to provide a precise relation between equivariant objects with associated non-equivariant objects over quotients; in Section \ref{s6} we give the proof of the main results of the paper, namely the interior product formulas, and of the functoriality of $\varrho$-classes for general surjective \'{e}tale maps of foliation groupoids, which generalizes \cite[Theorem 1.1]{GXY} and simplifies its proof; finally, in Section \ref{s7} we list the problems which remain unsolved and we propose further directions of research starting from this paper. 

\bigskip

 \section{Preliminaries}\label{s1}
\subsection{Definitions and notations}

\begin{definition} Let $G$ and $G^{(0)}$ be two sets.  A groupoid structure on $G$ over $G^{(0)}$ is given by the following morphisms:
	\begin{itemize}
		
		\item Two maps: $r,s: G\rightarrow G^{(0)}$,
		which are respectively the range and  source map.
		\item A map $u:G^{(0)}\rightarrow G$ called the unit map that is a section for both $s$ and $r$. We can identify $G^{(0)}$ with its image
		in $G$. 
		\item An involution: $ i: G\rightarrow G
		$, $  \gamma  \mapsto \gamma^{-1} $ called the inverse
		map. It satisfies: $s\circ i=r$.
		\item A map $ p: G^{(2)}  \rightarrow  G
		$, $ (\gamma_1,\gamma_2)  \mapsto  \gamma_1\cdot \gamma_2 $
		called the product, where the set 
		$$G^{(2)}:=\{(\gamma_1,\gamma_2)\in G\times G \ \vert \
		s(\gamma_1)=r(\gamma_2)\}$$ is the set of composable pair. Moreover for $(\gamma_1,\gamma_2)\in
		G^{(2)}$ we have $r(\gamma_1\cdot \gamma_2)=r(\gamma_1)$ and $s(\gamma_1\cdot \gamma_2)=s(\gamma_2)$.
	\end{itemize}
	
	The following properties must be fulfilled:
	\begin{itemize}
		\item The product is associative: for any $\gamma_1,\
		\gamma_2,\ \gamma_3$ in $G$ such that $s(\gamma_1)=r(\gamma_2)$ and
		$s(\gamma_2)=r(\gamma_3)$ the following equality
		holds $$(\gamma_1\cdot \gamma_2)\cdot \gamma_3= \gamma_1\cdot
		(\gamma_2\cdot \gamma_3)\ .$$
		\item For any $\gamma$ in $G$: $r(\gamma)\cdot
		\gamma=\gamma\cdot s(\gamma)=\gamma$ and $\gamma\cdot
		\gamma^{-1}=r(\gamma)$.
	\end{itemize}
	
	We denote a groupoid structure on $G$ over $G^{(0)}$ by
	$G\rightrightarrows G^{(0)}$,  where the arrows stand for the source
	and target maps. 
\end{definition}

We will adopt the following notations: $$G_A:=
s^{-1}(A)\ ,\ G^B=r^{-1}(B) \ \mbox{ and } G_A^B=G_A\cap G^B \,$$
in particular if $x\in G^{(0)}$, the  $s$-fiber (resp. 
$r$-fiber) of $G$ over $x$ is $G_x=s^{-1}(x)$ (resp. $G^x=r^{-1}(x)$).

\begin{definition}
	
	We call $G$ a Lie groupoid when $G$ and $G^{(0)}$ are second-countable smooth manifolds
	with $G^{(0)}$ Hausdorff, the structural homomorphisms are smooth.
\end{definition}

\bigskip

Let us now consider a commutative diagram of Lie
groupoids and vector bundles as follows:
\begin{equation}\label{vb}
\xymatrix{
	E\ar[d]\ar@<+0.4ex>[r]\ar@<-0.4ex>[r]& E^{(0)}\ar[d]\\
	G\ar@<+0.4ex>[r]\ar@<-0.4ex>[r]& G^{(0)}
	}
\end{equation}
Here we mean that $E\rightrightarrows E^{(0)}$ is a Lie groupoid (with source, target, multiplication,
identity, and inverse maps denoted by $\tilde{s}$, $\tilde{r}$, $\tilde{m}$, $\tilde{u}$, and $\tilde{i}$), $G\rightrightarrows G^{(0)}$ is a Lie groupoid (with source,
target, multiplication, identity, and inverse maps denoted by $s$, $t$, $m$, $u$, and $i$), $E\to G$ is a
vector bundle (with projection map and zero section denoted by  $\tilde{q}$ and $\tilde{0}$), $E^{(0)}\to G^{(0)}$ is a vector bundle bundle (with projection map and zero section  denoted by $q$ and $0$) and such that $q\tilde{s}=s\tilde{q}$ and
$q\tilde{t}=t\tilde{q}$.

\begin{definition}
	A $\mathcal{VB}$-groupoid is a commutative diagram of Lie groupoids and
vector bundles like (3.1) such that the following conditions hold:
\begin{enumerate}
\item $(\tilde{s}, s)$ and $(\tilde{t}, t)$ are  morphisms of vector bundles.
\item $(\tilde{q}, q)$ is a morphism of Lie groupoids.

\item Interchange law:
\[
(\eta_1+\eta_3)(\eta_2+ \eta_4)= \eta_1\eta_2 +\eta_3\eta_4
\] 
for any $(\eta_1, \eta_2), (\eta_3, \eta_4)\in E^{(2)}$ such that $\tilde{q}(\eta_1)=\tilde{q}(\eta_3)$ and $\tilde{q}(\eta_2)=\tilde{q}(\eta_4)$.
\end{enumerate}
\end{definition}

\begin{remark}
Recall that the usual definition includes the following technical
condition. Consider the vector bundle $s^*E^{(0)}=E^{(0)}\times_s G$ over $G$ 
and the map $p\colon E\to s^*E^{(0)}$.
The technical condition is that $p$ is required to be a surjective submersion.
But one can prove that this is implied by the definition. 
\end{remark}

\begin{definition}
Consider the sub-bundle  $\ker p\subset E$. Then define the (right) core of $E$ as the vector bundle over $G^{(0)}$ given by $C:=u^*(\ker p)$.   
\end{definition}

For any pair $(\gamma_1,\gamma_2)\in G^{(2)}$, right-multiplication  by $\tilde{0}(\gamma_2)$ produces a linear isomorphism from $(\ker p)_{\gamma_1}$ to $(\ker p)_{\gamma_1\gamma_2}$. In particular, for any $\gamma\in G$, right-multiplication by $\tilde{0}(\gamma)$ produces a linear isomorphism from $C_{r(\gamma)}$ to $(\ker p)_{\gamma}$. Hence we have a natural isomorphism of vector bundles over $G$ between $r^* C$ to $\ker p$ given by $(\gamma,c)\mapsto c\cdot \tilde{0}(\gamma)$. 

So, for any section $h\colon s^*E^{(0)}\to E$, we have an isomorphism of vector bundles \begin{equation}\label{rh-iso} s^*E^{(0)}\oplus r^*C\cong E \quad \mbox{given by} \quad (\gamma, e\oplus c)\mapsto h(e) + c\cdot \tilde{0}(\gamma).\end{equation}

\begin{definition}
	A \emph{right-horizontal lift} of the $\mathcal{VB}$-groupoid $E\rightrightarrows E^{(0)}$ is given by a section $h\colon s^*E^{(0)}\to E$ of $p$ such that
	$$h(u(x), e)= \tilde{u}(e)\quad \forall x\in G^{(0)}\, \mbox{and}\, e\in E^{(0)}_x.$$
\end{definition}
Observe that a right-horizontal lift $h$ is such that the associated isomorphism \eqref{rh-iso} restricts to the natural splitting 
\[
E_{u(x)}\cong C_x\oplus \tilde{u}(E^{(0)}_x)\quad \forall x\in G^{(0)}.
\]
\begin{remark}\label{affine}
It is important to point out that, by \cite[Lemma 5.9]{GS-AM},
	the space of  right-horizontal lifts  of a $\mathcal{VB}$-groupoid is an affine space.
\end{remark}

\begin{remark}\label{rmk-core}
	Observe that the core bundle $C\to G^{(0)}$ is a $G$-bundle and the action of $\gamma\in G$ which sends $C_{r(\gamma)}\to C_{s(\gamma)}$ is given by 
	$c\mapsto  \tilde{0}(\gamma^{-1})\cdot c\cdot \tilde{0}(\gamma)$.
\end{remark}

\subsection{The deformation to the normal cone}

Let $\iota\colon M_0\to M$ be an immersion of  smooth manifolds, with normal bundle $N_\iota$.
As a set, the deformation to the normal cone is \begin{equation}
DNC^{\RR}(M, M_0):=N_\iota \times\{0\}\sqcup M\times\RR^*.
\end{equation}

\noindent\textbf{Notation:} \emph{we will consider restrictions of the $DNC$ to various subsets of $\RR$, in particular we will denote by $DNC$, without any decoration, the restriction to $[0,1]$ and by $DNC^{[0,1)}$ the restriction to the interval $[0,1)$ open at $1$.}

In order to recall its smooth structure, we fix an exponential map, which is a diffeomorphism $\theta$ from
a neighbourhood $V'$ of the zero section $M_0$ in $N$ to a neighbourhood $V$ of $M_0$ in $M$.
We may cover $DNC(M,M_0)$ with two open sets $M\times \RR^*$, with the product differentiable structure, and $W=N\times {0}\sqcup V\times\RR^*$,
endowed with the differentiable structure for which the map 
\begin{equation}\label{adiabatictopology}
\Psi\colon\{(m,\xi,t)\in N\times\RR\,|\,(m,t\xi)\in V'\}\to W
\end{equation}
given by $(m,\xi,t)\mapsto(\theta(m, t\xi),t)$,
for $t\neq0$, and by $(m,\xi,0)\mapsto(m,\xi,0)$, for $t=0$, is a diffeomorphism. One can verify that the transition map on the overlap of these two charts is smooth, see for instance \cite[Section 3.1]{HSk}.

Let us consider a commutative square of the following type
\begin{equation}\label{dnc-immersion}
\xymatrix{M_0\ar[r]^\iota\ar[d]_{f_0}& M\ar[d]^{f}\\
	            M_0'\ar[r]^{\iota'}& M'}
\end{equation}
where $\iota$ and $\iota'$ are immersions, $f$ and $f_0$ are smooth maps. 
Then we obtain the smooth map $DNC(f,f_0)\colon DNC(M, M_0)\to DNC(M',M_0')$, defined by 
\begin{equation*}
\begin{split}
DNC(f,f_0)(x,t)&=(f(x),t)\quad \mbox{for}\, t\neq0,\\
DNC(f,f_0)(x,\xi,0)&=(f(x),df_x(\xi),0)\quad \mbox{for}\, t=0.
\end{split}
\end{equation*}

\begin{remark}\label{imm-sub}
Observe that, as for instance in \cite[Proposition 1.1]{Mohsen}, $DNC(f,f_0)$ is
\begin{itemize}
\item 	a submersion if and only if both $f$ and $f_0$ are submersions,
\item   an immersion if and only if $f$ is an immersion and $T_xM_0=(df_0)_x^{-1}(T_{f_0(x)}M'_0)$.
\end{itemize}
When there is no risk of confusion, we will denote $DNC(f,f_0)$ just by $DNC(f)$.
\end{remark}

\bigskip

\begin{definition} \label{def-dnc}
	Let $\iota\colon H\looparrowright G$ be a smooth immersion of Lie groupoids and let us denote its restriction to the units by $\iota_0\colon H^{(0)}\hookrightarrow G^{(0)}$.  
	\begin{enumerate}
\item The normal bundle $N_\iota$ carries a Lie groupoid structure over $N_{\iota_0}$. 
We denote this groupoid by  $\mathcal{N}_\iota\rightrightarrows N_{\iota_0} $. We will also use the following notation $\mathcal{N}_H^G$ to denote $\mathcal{N}_\iota$.
\item The smooth manifold $DNC(G,H)$ is naturally endowed with a structure of Lie groupoid over $DNC(G^{(0)},H^{(0)})$. The source and range morphisms are given by $DNC(s)$ and $DNC(r)$ respectively. Furthermore, the set of composable arrows $DNC(G,H)^{(2)}$ identifies with $DNC(G^{(2)},H^{(2)})$ and the product is given by $DNC(m^G,m^H)$.
\item For the deformation to the normal cone associated with the unit morphism $u\colon G^{(0)}\to G$ we will use the following notation $G_{ad}:=DNC(G,G^{(0)} )\rightrightarrows G^{(0)}\times[0,1]$ and we will call it the \emph{adiabatic groupoid} associated with $G$. The same notation applies to morphisms, namely $f_{ad}:=DNC(f)$. Finally, recall that in this particular case  $\mathcal{N}_u$ is just $\mathfrak{A}G$, the Lie algebroid of $G$.
\end{enumerate}  
\end{definition}

\begin{remark}
	Notice that the projection $\mathcal{N}_\iota\to H$ is a groupoid morphism and it follows that $\mathcal{N}_\iota$ is a $\mathcal{VB}$-groupoid over $H$. 
\end{remark}

\subsection{Groupoid C*-algebras}
We can associate with a Lie groupoid $G$ the *-algebra $C^\infty_c(G,\Omega^{\frac{1}{2}}(\ker ds\oplus\ker dr))$ of the compactly supported sections of the half densities bundle associated with $\ker ds\oplus\ker dr$, with:

\begin{itemize}
	\item the involution given by $f^*(\gamma)=\overline{f(\gamma^{-1})}$;
	\item and the convolution product given by $f*g(\gamma)=\int_{G_{s(\gamma)}} f(\gamma\eta^{-1})g(\eta)$.
\end{itemize}  

For all $x\in G^{(0)}$ the algebra $C^\infty_c(G,\Omega^{\frac{1}{2}}(\ker ds\oplus\ker dr))$ can be represented on 
$L^2(G_x,\Omega^{\frac{1}{2}}(G_x))$ by 
\[\lambda_x(f)\xi(\gamma)=\int_{G_{x}} f(\gamma\eta^{-1})\xi(\eta), \]
where $f\in C^\infty_c(G,\Omega^{\frac{1}{2}}(\ker ds\oplus\ker dr))$ and $\xi\in L^2(G_x,\Omega^{\frac{1}{2}}(G_x))$.

\begin{definition}
	The reduced C*-algebra of a Lie groupoid G, denoted by $C^*_r(G)$, is the completion of $C^\infty_c(G,\Omega^{\frac{1}{2}}(\ker ds\oplus\ker dr))$ with respect to the norm
	\[
	||f||_r=\sup_{x\in G^{(0)}}||\lambda_x(f)||.
	\]
	
	The full C*-algebra of $G$ is the completion of 
	$C^\infty_c(G,\Omega^{\frac{1}{2}}(\ker ds\oplus\ker dr))$ with respect to all continuous representations.
\end{definition}

\begin{remark}\label{fullvsred}
	From now on, if $X$ is a $G$-invariant closed subset of $G^{(0)}$ we will call
	$e_X\colon C^\infty_c(G)\to C^\infty_c(G_{|X})$ the restriction map to $X$.
	That gives an exact sequence of full groupoid C*-algebras
	\[
	\xymatrix{0\ar[r]& C^*(G_{|G^{(0)}\setminus X})\ar[r]&C^*(G)\ar[r]& C^*(G_{|X})\ar[r]&0},
	\]
	but in general this is not true for the reduced ones:  
	the reader can find examples of this phenomenon in \cite{HLS}.
	Let us precise that in what follows we will mainly deal with the reduced groupoid C*-algebras, because there are more details to check in the reduced situation. But everything we are going to  prove about the reduced C*-algebras works for the full C*-algebras too.
\end{remark}

\section{Some KK-elements}\label{s2}
This section is dedicated to recall the technical constructions in KK-theory which will be used all along this paper.

\subsection{Mapping cones and KK-classes}

\noindent\textbf{Degenerate bimodules.} Recall that a Kasparov bimodule $(\mathcal{H},F)\in \mathbb{E}(A, B)$ is said to be degenerate if 
$$[a,F]=a(F^2-1)=a(F-F^*)=0,\quad \forall a\in A$$
and that $\mathbb{D}(A, B)$ denotes the set of all
degenerate Kasparov $A,B$-bimodules.
In  this case $(\mathcal{H}\otimes C_0[0,1), F\otimes 1)\in \mathbb{E}(A, B[0,1])$ produces a canonical homotopy from $(\mathcal{H},F)\in \mathbb{E}(A, B)$ to the zero bimodule $(0,0)$. Namely $\mathbb{D}(A, B)$  is star-shaped.
Thus, from now on, we will do a little abuse by considering any degenerate bimodule as the zero bimodule. In particular, a bimodule $(\overline{\mathcal{H}},\overline{F})\in \mathbb{E}(A,B[0,1])$ such that $(\mathrm{ev}_0)_*(\overline{\mathcal{H}},\overline{F})$ is degenerate will define in a canonical way an element in $\mathbb{E}(A,B[0,1))$.

\medskip

\noindent\textbf{Gluing bimodules \cite[Section 2.3]{AAS}.}
Let $\varphi_0\colon B_0\to C$ and $\varphi_1\colon B_1\to C$ be two *-homomorphisms. We can then construct the pull-back C*-algebra $$B_0\oplus_C B_1:=\{(b_0,b_1)\in B_0\oplus B_1 | \varphi_0(b_0)=\varphi_1(b_1)\}.$$
Consider now two bimodules $(\mathcal{H},F)\in \mathbb{E}(A,B_0)$ and $(\mathcal{H}',F')\in \mathbb{E}(A,B_1)$. Let us assume that $\varphi_0$ and $\varphi_1$ are surjective *-homomorphism and recall that this implies that, for instance, $\mathcal{H}\otimes_{\varphi_0}C$ is a quotient of $\mathcal{H}$ (let us keep the notation $\varphi_0$ for the quotient map).
\begin{definition}
	Let $w$ be a unitary equivalence between $(\varphi_0)_*(\mathcal{H},F)$ and $(\varphi_1)_*(\mathcal{H}',F')$. Then define the Kasparov bimodule
	\begin{equation}\label{diamond}
	(\mathcal{H},F)\diamond^w_{\varphi_0,\varphi_1} (\mathcal{H}',F')\in \mathbb{E}(A, B_0\oplus_C B_1)
	\end{equation}
 to be given by the pair
 \begin{itemize}
\item $\mathcal{H}\diamond \mathcal{H}':=\{(\eta,\xi)\in\mathcal{H}\times \mathcal{H}'\,|\,\varphi_0(\xi)=w\varphi_1(\eta) \}$,
\item $F\diamond F'(\eta,\xi):=(F\eta,F'\xi)$.
 \end{itemize}
\end{definition}

\medskip

\noindent\textbf{Mapping cones.}
Let $\varphi\colon A\to B$ be a *-homomorphism. Associated with $\varphi$ we have the mapping cylinder C*-algebra 
$$\mathcal{Z}_{\varphi}=:\{(a,f)\in A\oplus B[0,1]\,|\, \varphi(a)=f(0)\}$$ 
and the mapping cone C*-algebra 
$$\mathcal{C}_{\varphi}=:\{(a,f)\in A\oplus B[0,1)\,|\, \varphi(a)=f(0)\}$$  
which is an ideal in $\mathcal{Z}_{\varphi}$.

A typical element in $\mathbb{E}(D,\mathcal{C}_{\varphi})$ is given by a pair \begin{equation}\label{mc-class}
\left((\mathcal{H}_{A}, G_{A}), (\overline{\mathcal{H}},\overline{G})\right), \end{equation}
 composed by a bimodule $(\mathcal{H}_A, G_A)\in \mathbb{E}(D,A)$ and a bimodule $(\overline{\mathcal{H}},\overline{G})\in \mathbb{E}(D,B[0,1])$ which is a homotopy between $\varphi_*(\mathcal{H}_A, G_A)$ and  a degenerate bimodule $(\mathcal{H}_B, G_B)\in \mathbb{D}(D, B)$. 
This fits with \eqref{diamond} where $w$ is the identity of $B$.

\begin{remark}
	One could argue that a typical element in $\mathbb{E}(D,\mathcal{C}_{\varphi})$ is defined by means of a homotopy between $\varphi_*(\mathcal{H}_A, G_A)$ and the zero module, but by gluing this homotopy with the canonical homotopy $(\mathcal{H}_B\otimes C_0[0,1), G_B\otimes 1)$ we establish a canonical equivalence between these definitions. We keep the degenerate bimodule at the end just because it allows more flexibility with Kasparov products as we will see. 
\end{remark}

Let us consider two *-homomorphisms $\varphi_i\colon A_i\to B_i$, with $i=0,1$, and  two bimodules $(\mathcal{E}_A,F_A)\in \mathbb{E}(A_0,A_1)$ and $(\mathcal{E}_B, F_B)\in \mathbb{E}(B_0, B_1)$ such that there exists a homotopy $(\overline{\mathcal{E}}, \overline{F})$ between $(\varphi_1)_*(\mathcal{E}_A,F_A)$ and $\varphi_0^*(\mathcal{E}_B, F_B)$ in $\mathbb{E}(A_0,B_1[0,1])$. Then this set of data produces a bimodule $(\mathcal{E}_\varphi,F_\varphi)\in \mathbb{E}(\mathcal{C}_{\varphi_0},\mathcal{C}_{\varphi_1})$, see \cite[Lemma 5.2]{An-Sk} for a detailed construction.

Let us see how we will intend in practice the Kasparov product of the  element $[(\mathcal{H}_{A_0}, G_{A_0}), (\overline{\mathcal{H}},\overline{G})]\in KK(D,\mathcal{C}_{\varphi_0})$ with $[\mathcal{E}_\varphi,F_\varphi]\in KK(\mathcal{C}_{\varphi_0},\mathcal{C}_{\varphi_1})$:
\begin{itemize}
\item let $(\mathcal{H}_{A_1},G_{A_1})\in \mathbb{E}(D,A_1)$ be a Kasparov product of $(\mathcal{H}_{A_0},G_{A_0})$ and $(\mathcal{E}_A,F_A)$;
\item let $(\overline{\mathcal{H}}',\overline{G}')\in \mathbb{E}(D, B_1[0,1])$ be a Kasparov product of $(\overline{\mathcal{H}},\overline{G})$ and $(\overline{\mathcal{E}}, \overline{F})$ such that the evaluation at $0$ of $(\overline{\mathcal{H}}',\overline{G}')$ is equal to $(\varphi_1)_*(\mathcal{H}_{A_1},G_{A_1})$;
\item observe that $(\mathcal{H}_{B_1},G_{B_1})$, the evaluation at $1$ of $(\overline{\mathcal{H}}',\overline{G}')$, is a Kasparov product of the degenerate bimodule $(\mathcal{H}_B, G_B)$ and $(\mathcal{E}_B, F_B)$. This is not necessarily degenerate, so the pair $\left((\mathcal{H}_{A_1},G_{A_1}),(\overline{\mathcal{H}}',\overline{G}')\right)$ only defines an element in $\mathbb{E}(D,\mathcal{Z}_{\varphi_1})$;
\item  but Lemma \ref{lemma-cs} below tells us that we have a canonical operatorial homotopy of $(\mathcal{H}_{B_1},G_{B_1})$ to the degenerate bimodule $(\mathcal{H}_{B_0}\otimes\mathcal{E}_B, G_{B_0}\otimes 1)$ by means of the operatorial homotopy $CS_t(G_{B_1},G_{B_0}\otimes1)$ defined below in \eqref{cst}. So the element in $\mathbb{E}(D,\mathcal{Z}_{\varphi_1})$, defined in the previous point, lifts canonically to the element 
\begin{equation}\label{k-product-mc}
\left((\mathcal{H}_{A_1},G_{A_1}),(\overline{\mathcal{H}}',\overline{G}')\diamond^{\mathrm{id}}_{\mathrm{ev}_1,\mathrm{ev}_0}(\mathcal{H}_{B_1},CS_t(G_{B_1},G_{B_0}\otimes1))\right)\in\mathbb{E}(D,\mathcal{C}_{\varphi_1})
\end{equation}
which will be our canonical representative for a Kasparov product of this kind.
\end{itemize}

\begin{lemma}{\cite[Lemma 11]{Sk-remarks}}\label{lemma-cs}
	Let $(\mathcal{E}, F)$ and $(\mathcal{E}, F')$ be two elements of $\mathbb{E}(A,B)$ such that $\forall a\in A$ we have that $a[F,F']a^*-\lambda\geq0$ modulo $\mathbb{K}(\mathcal{E})$, with $\lambda>-2$. Then $(\mathcal{E}, F)$ and $(\mathcal{E}, F')$ are operatorially homotopic. 
\end{lemma}

\begin{proof}
	Let us define the algebra
	$$\mathcal{A}:=\{T\in \mathbb{B}(\mathcal{E})\,|\, [T,a]\in \mathbb{K}({\mathcal{E}})\, ,\forall a\in A \}$$
	and its ideal
	$$\mathcal{J}:=\{T\in \mathcal{A}\,|\, Ta\in \mathbb{K}({\mathcal{E}}) \, ,\forall a\in A\}.$$
	Then $[F,F']\in\mathcal{A}$ and it is positive modulo $\mathcal{J}$. Write $[F,F']=P+K$, where $P\in \mathcal{A}$, $P\geq\lambda$ and $K\in \mathcal{J}$. Note that as $F^2-1$ and $F'-1$ belong to $\mathcal{J}$ then $[F,P]$ and $[F',P]$ belong to $\mathcal{J}$. 
	Put 
	\begin{equation}\label{cst}
   CS_t(F,F'):= (1+\sin(t)\cdot \cos(t) P)^{-\frac{1}{2}}(\cos(t)F+\sin(t)F')\quad t\in[0,\pi/2].
	\end{equation}
Therefore, it needs a simple calculation to check that $(\mathcal{E}, CS_t(F,F'))\in \mathbb{E}(A,B[0,1])$ and then it realizes the desired operatorial homotopy.
	\end{proof}
The well-known proof of the previous lemma is useful in order to prove the following one, which will be key for our main result. 

\begin{lemma}\label{lemma-cs-degenerate}
		Let $F$ and $F'$ be regular operators on $\mathcal{E}$ such that $F$ and $F'$ are invertible and  that $[F,F']-\lambda\geq0$ in 
		$\mathbb{B}(\mathcal{E})$, with $\lambda>-2$. Then $F$ and $F'$ are homotopic through a path of invertible operators. 
\end{lemma}

\begin{proof}
Write in this case $[F,F']=P$. Then a simple calculation shows that when $\lambda\in [0,-2)$ we have that
$$(\cos(t)F+\sin(t)F')^2= 1+\sin(t)\cdot \cos(t)P\geq1+\frac{\lambda}{2}>0,$$
while when $\lambda$ is positive, we directly obtain that $(\cos(t)F+\sin(t)F')^2>0$. 
This implies that   $(\cos(t)F+\sin(t)F')$ has a gap near zero in the spectrum and then the same holds for $CS_t(F,F')$.
Thus $(\mathcal{E}, \mathrm{sgn}(CS_t(F,F')))$ realizes the desired homotopy through invertible operators.  
\end{proof}

\subsection{Bott bimodules}

In this subsection we recall the constructions of Bott KK-elements given in \cite[Section 5]{Kasparov81}.
Let $Cliff(\RR^n)$ be the Clifford algebra associated with the euclidean metric on $\RR^n$. 
Consider the $\RR,C^{\RR}_0(\RR^n)\otimes Cliff(\RR^n)$-bimodule  $C^{\RR}_0(\RR^n)\otimes Cliff(\RR^n)$.  The bounded multiplier $F:=x\mapsto x\cdot(1+\|x\|^2)^{-1}$ is $GL_n(\RR)$-invariant and we have that $F^*=F$ and that $1-F^2=(1+\|x\|^2)^{-\frac{1}{2}}\in C^{\RR}_0(\RR^n)\otimes Cliff(\RR^n)$. 

\begin{definition}\label{bott1}
Let us denote by $\mathcal{B}\in KKO_n^{GL_n(\RR)}(\RR, C^{\RR}_0(\RR^n)\otimes Cliff(\RR^n))$ the element given by 
\begin{equation}[C^{\RR}_0(\RR^n)\otimes Cliff(\RR^n), F].
\end{equation}
\end{definition}

Let $\mathfrak{j}\colon \mathcal{H}\to GL_n(\RR)$ be an homomorphism of groups and let $S$ be an  $\mathcal{H}$-equivariant finite dimensional representation of $Cliff(\RR^n)$. This naturally gives an element of $KK^\mathcal{H}(Cliff(\RR^n), \CC)$ in the following way.

\begin{definition}\label{bott2}
	Let us denote by $\beta^{\mathfrak{j}}_n(S)\in KKO^\mathcal{H}_n(\RR, C^{\RR}_0(\RR^n))$ the element given by 
	\begin{equation}
\mathfrak{j}^*(\mathcal{B})\otimes_{Cliff(\RR^n)} [S].
    \end{equation}
\end{definition}

When $\mathfrak{j}$ is the natural inclusion of a subgroup $\mathcal{H}$ or, for instance, the natural map from $Spin$ or from the metalinear group  $Ml$  to $GL$, we will  denote the Bott element associated with $S$ by $\beta_n^{\mathcal{H}}(S)$.

\begin{remark}
Notice  that, after complexifying, one can construct analogous elements in $KK^\mathcal{H}_n(\CC, C^{\RR}_0(\RR^n))$ associated with a complex representation of the complexified Clifford algebra $\CC liff(\RR^n):= Cliff(\RR^n)\otimes \CC$. 
\end{remark}

\begin{example}
The most frequent geometrical examples of  Bott  elements are the following ones: 
\begin{enumerate}
	\item \textbf{The oriented case:} consider  the representation of $\CC liff(\RR^n)$ given by the complexified exterior algebra $\Lambda^*_{\CC}\RR^n$ of $\RR^n$. Then we have the class
	$$\beta^{SO_n}_n(\Lambda_\CC\RR^n):=\left[C_0(\RR^n, \Lambda^*_{\CC}\RR^n), x\mapsto \frac{\lambda_x+\lambda_x^*}{\sqrt{1+\|x\|^2}}\right]\in KK_n^{SO_n}(\CC,C_0(\RR^n)),$$
%
	where  $\lambda_x$ is the exterior multiplication by $x\in \RR^n$. 
%
	
	\item \textbf{The $KO$-oriented case:} 
	consider the real spinor representation $\slashed{S}^n$ of $Cliff(\RR^n)$, then we obtain the class $$\slashed{\beta}_n:= \left[C^{\RR}_0(\RR^n, \slashed{S}^n), x\mapsto \frac{cl(x)}{\sqrt{1+\|x\|^2}}\right]\in KKO_n^{Spin_n}(\RR,C^{\RR}_0(\RR^n)).$$ 
	
	
	\item \textbf{The $K$-oriented case:}  $\slashed{\beta}_n^c\in KK_n^{Spin^c_n}(\CC,C_0(\RR^n))$
 is defined similarly to the $KO$-oriented case. 
\end{enumerate}
\end{example}

\subsection{KK-elements associated with groupoid cocycles}

In this section we are going to quickly recall some constructions from \cite[Section II]{HSk}, concerning Thom elements in KK-theory. 

Let $\mathcal{H}$ be a Lie group and let $i\colon G\to \mathcal{H} $ be a cocycle on  the Lie groupoid $G\rightrightarrows M$.
More precisely it is given by the following data:
\begin{itemize}
	\item an open cover $\{U_j\}_{j\in J}$ of $M$;
	\item for all $j,k\in J$ an application $i_{j,k}\colon G^j_k\to \mathcal{H}$ such that 
	$$i_{j,h}(\gamma_1\gamma_2)=i_{j,k}(\gamma_1)i_{k,h}(\gamma_2)$$
	for all $\gamma_1\in G^j_k$ and $\gamma_2\in G^k_h$.  Here $ G^j_k$ denotes $G^{U_j}_{U_k}$.
\end{itemize}

It is equivalent to consider the principal $G$-equivariant fiber bundle $P_i$ with structural group $\mathcal{H}$, 
constructed via the cocycle $i_0$, which is the restriction of $i$ to $M$ (recall that in this situation the actions of $\mathcal{H}$ and $G$ on $P_i$ commute).

Now, if $A$ is an $\mathcal{H}$-algebra, the associated bundle $P_i\times_{\mathcal{H}} A$ is $G$-equivariant and we can construct the 
crossed product C*-algebra $A\rtimes_i G:=C^*(G; P_i\times_{\mathcal{H}} A)$. 
Notice that, if we denote by $G^{P_i}$ the transformation groupoid $P_i\rtimes G\rightrightarrows P_i$  associated with $i$, then $A\rtimes_i G$ is 
 Morita equivalent to $(A\otimes C^*(G^{P_i}) )\rtimes \mathcal{H}$, let us denote by $M_{A,i}$ the bimodule associated with this Morita equivalence.

\begin{definition}
	Let $i\colon G\to \mathcal{H}$ be a cocycle as before. Let $A$ and $B$ be $\mathcal{H}$-algebras, then set the morphism of KK-groups
	$$ i^*\colon KK^{\mathcal{H}}(A,B)\to KK(A\rtimes_i G,B\rtimes_i G)$$
by 
$$x\mapsto M_{A,i}\otimes j^{\mathcal{H}}(\mathrm{Id}_{C^*(G^{P_i})}\otimes_\CC x)\otimes M^{-1}_{B,i}$$
\end{definition}
Intuitively, $i^*$ assembles an equivariant bundle over $G$ whose typical fiber is given by the Kasparov bimodule $x$.


	Let $E\xrightarrow{\pi}G^{(0)}$ be a $G$-equivariant vector bundle, namely it is associated with the restriction to $G^{(0)}$ of a cocycle $i\colon G\to \mathcal{H}$, where $\mathcal{H}$ is the structural Lie group of $E$. 
	Then we can consider the action groupoid $G^E:=E\rtimes G\rightrightarrows E$. As a set, $G^E$ is equal to the pull-back bundle $r^*E$ over $G$.
	We have the following isomorphisms of C*-algebras
	$$C^*(G^E)\cong C_0(E)\rtimes G\cong C_0(P_i\times_{\mathcal{H}} \RR^n)\rtimes G\cong C^*(G; P_i\times_{\mathcal{H}} C_0(\RR^n))=: C_0(\RR^n)\rtimes_i G.$$ 
 
\begin{definition}\label{bott3}
	Let $E\to G$ be a vector bundle  of rank $n$ associated with a cocycle $i\colon G\to \mathcal{H}$. Let $\mathcal{S}$ be a representation of $Cliff(\RR^n)$ which assembles to a  $Cliff(E)$-module  $S$ over $G$.
	Then define the class $\beta(E,S)$ as the KK-element $i^*(\beta^\mathcal{H}_n(\mathcal{S}))$ in the KK-groups $KK_n(C^*(G), C^*(G^E))$, in the complex case,
	or in $KKO_n(C^*_{\RR}(G), C^*_{\RR}(G^E))$, in the real case.
\end{definition}

\begin{remark}\label{HS-normal}
The usual case where we want to apply this construction is the one of $\mathcal{N}_\iota$, the normal bundle of an immersion $\iota\colon H\to G$ of Lie groupoids. It is a $\mathcal{VB}$-groupoid with objects $N_{\iota_0}$ and its C*-algebra $C^*(\mathcal{N}_\iota)$ is not always identifiable in the form $C^*(H^E)$, for some $H$-bundle $E$ over $H^{(0)}$. 

Examples where this happens are the following ones:
\begin{enumerate}
	\item  if $N_{\iota_0}$  is trivial, then $C^*(\mathcal{N}_\iota)$ is isomorphic to $C^*(H^C)$ via fiber-wise Fourier transform, where $C$ is the core of $\mathcal{N}_\iota$ and, by Remark \ref{rmk-core}, it is an $H$-bundle.
	\item More generally, assume that $N_{\iota_0}$ is an $H$-bundle such that $s^*N_{\iota_0}= r^*N_{\iota_0}$. Then, by \eqref{rh-iso},  we have that
	$C^*(\mathcal{N}_\iota)\cong C^*(H^{N_{\iota_0}\oplus C})$. 
	Observe that this isomorphism uses the fiber-wise Fourier transform along $C$, but thanks to Remark \ref{affine}, the isomorphism induced in KK-theory does not depend on the splitting used in \eqref{rh-iso}.
\end{enumerate}
\noindent\emph{In this case we will shortly denote  $\beta(\mathcal{N}_\iota,S)$ just by $\beta(\iota,S)$ and suppress $S$ from the notation when it is obvious from the context.}

\end{remark}

\subsection{The asymptotic morphism associated with a DNC}\label{DNC-boundary}

Let us fix an immersion of Lie groupoids $\iota\colon G\to H$ and consider the associated deformation to the normal cone 
$DNC(G,H)\rightrightarrows DNC(G^{(0)},H^{(0)})$. Its C*-algebra is a $C[0,1]$-algebra and it comes with the following extension of $C[0,1]$-algebras
\begin{equation}\label{DNC-asymp}
\xymatrix{0\ar[r]& C^*(G)\otimes C_0(0,1]\ar[r]& C^*(DNC(G,H))\ar[r]^(.6){\mathrm{ev}_0}& C^*(\mathcal{N}_H^G)\ar[r]&0}
\end{equation}
where $C[0,1]$ acts on $C^*(\mathcal{N}_H^G)$ by evaluation at 0.
Notice that $C^*(G)\otimes C_0(0,1]$ is contractible and $\mathrm{ev}_0$ induces an invertible asymptotic morphism. 
Hence, as it is explained in \cite[Lemma A.2]{kasp-sk}, we can give the following definition. 

\begin{definition}\label{E-dnc} The exact sequence \eqref{DNC-asymp} gives rise to an element in the E-theory group $E(C^*(\mathcal{N}_H^G), C^*(G))$ which is explicitly obtained as 
	\begin{equation}\label{partial}
	\partial_H^G:=1_{C^*(\mathcal{N}_H^G)}\otimes[\mathrm{ev}_{0}]^{-1}\otimes [\mathrm{ev}_{1}].
	\end{equation}
	\noindent \emph{In some case, where the notation would be too heavy, we will denote it by $\partial(\iota)$.}
\end{definition}
If, in addition, $\mathrm{ev}_{0}$ has a completely positive lifting of norm 1, then  $\mathrm{ev}_{0}$ induces a KK-equivalence $KK(C^*(\mathcal{N}_H^G),C^*(DNC(G,H)))\xrightarrow{[\mathrm{ev}_0]}KK(C^*(\mathcal{N}_H^G),C^*(\mathcal{N}_H^G))$ and $\partial_H^G$ is an element of $KK(C^*(\mathcal{N}_H^G), C^*(G))$. 

\begin{remark}
	Let us point out the obvious fact that one can see $\partial_H^{G}$ as the boundary map associated with the short exact sequence
	\begin{equation}\label{DNC-es}
	\xymatrix{0\ar[r]& C^*(G)\otimes C_0(0,1)\ar[r]& C^*(DNC^{[0,1)}(G,H))\ar[r]^(.6){\mathrm{ev}_0}& C^*(\mathcal{N}_H^G)\ar[r]&0}.
	\end{equation}

	Observe that $C^*(DNC^{[0,1)}(G,H))$ is isomorphic to the mapping cone C*-algebra of $\mathrm{ev}_1\colon C^*(DNC(G,H))\to C^*(G)$, then \eqref{DNC-es} is equivalent in K-theory to the exact sequence
	\begin{equation}
	\xymatrix{0\ar[r]& C^*(G)\otimes C_0(0,1)\ar[r]& C^*(DNC^{[0,1)}(G,H))\ar[r]& C^*(DNC(G,H))\ar[r]&0}
	\end{equation}
	with associated boundary map given by the composition of  $[\mathrm{ev}_1]$ and the suspension map.
	
	Thus, an element in $\mathbb{E}(\CC,C^*(DNC^{[0,1)}(G,H)))$ is given by a $\CC,C^*(DNC^{[0,1]}(G,H))$-bimodule $(\mathcal{E},F)$ and a homotopy in $\mathbb{E}(\CC,C^*(G))$ of $(\mathrm{ev}_1)_*(\mathcal{E},F)$ to a degenerate bimodule. 
\end{remark}

\section{Longitudinal classes}\label{longitudinal-classes}
In this section we shall recall the construction of secondary invariants given in \cite{zenobi-ad}.
The geometrical setting is the following one where, to keep the presentation simple, we are going to present only the complex case, but everything works analogously in the real case with KO-theory.
\begin{itemize}
	\item Let $G\rightrightarrows M$ be a Lie groupoid such that its Lie algebroid $\mathfrak{A}G$ is of rank $n$ and let us assume that the anchor map is injective (i.e. $\mathfrak{A}G$ is an involutive sub-bundle of the tangent bundle $TM$).
	
	\item Let $g$ be a metric on $\mathfrak{A}G$, by means of it we can define a $G$-invariant metric on $\ker{ds}$ along the $s$-fibers of $G$.
	\item Let $\mathrm{Cliff}_g\left(\mathfrak{A}G\right)$ be the Clifford algebra bundle over $M$ associated with the metric $g$. Let us fix $E\to M$, an hermitian bundle of  $\mathrm{Cliff}_g(\mathfrak{A}G)$-modules and let  $cl(X)$ denote the Clifford multiplication by $X\in\mathrm{Cliff}_g\left(\mathfrak{A}G\right)$.
	\item We will denote by $\mathcal{E}(G)$ the $C^*(G)$-module obtained as the completion  of the module $C^\infty_c(G, r^*E\otimes \Omega^{\frac{1}{2}}(G))$ with respect to the usual $C^*(G)$-valued inner product.
	\item Let $\nabla$ denote the fiberwise Levi-Civita  connection  associated with the metric $g$.
\end{itemize}
Assume that $E$ is equipped with a metric $h$ and a compatible connection $\nabla^E$ such that:
\begin{itemize}
	\item the Clifford multiplication is skew-symmetric, that is
	\[
	\langle cl(X)s_1,s_2\rangle+\langle s_1,cl(X)s_2\rangle=0
	\]
	for all $X\in C^\infty\left(M,\mathfrak{A}G\right)$ and $s_1,s_2\in C^\infty(M,E)$;
	\item $\nabla^E$ is compatible with the Levi-Civita connection $\nabla$, namely
	\[
	\nabla^E_X(cl(Y)s)=cl(\nabla_XY)s+cl(Y)\nabla^E_X(s)
	\]
	for all $X,Y\in C^\infty\left(M,\mathfrak{A}G\right)$ and $s\in  C^\infty(M,E)$.
\end{itemize}
\begin{definition}\label{diracgr}
	The generalized Dirac operator associated with this set of data is defined as
	\[
	D_G^E\colon s\mapsto \sum_{\alpha}c(e_\alpha)\nabla^E_\alpha(s)
	\]
	for $s\in C_c^\infty(G,r^*E_G\otimes \Omega^{\frac{1}{2}}(G))$ and $\{e_\alpha\}_{\alpha\in A}$ a local orthonormal frame. Here, by a little abuse of notation, we still denote by $\nabla^E$ the pull-back of the connection to $r^*E$.
	
	Let $p\colon \mathfrak{A}^*G\to M$ be the bundle projection, then the symbol of $D_G^E$ is given by the 
	section $\sigma_E\in C^\infty( \mathfrak{A}^*G, \mathrm{End}(p^*E))$, defined by $\sigma_E\colon \xi\mapsto c(\xi)$.
	
\end{definition}

\begin{examples}\label{main-ex}
	The typical geometrical examples for $E$ are the following ones:
	\begin{enumerate}
		\item if $\mathfrak{A}G$ is orientable, then for $E$ equal to $\Lambda^*(\mathfrak{A}G)$, the exterior algebra of the Lie algebroid,  $D_G^{\Lambda^*}$ is equal to $D^{sign}_G$, the longitudinal Signature operator on $G$;
		\item if $\mathfrak{A}G$ is $Spin$ or $Spin^c$, then for $E$ equal to $\slashed{S}_G$, the spinor bundle associated with the $Spin$ or $Spin^c$ structure of the Lie algebroid,  $D_G^{\slashed{S}}$ is equal to $\slashed{D}_G$, the longitudinal $Spin$ or $Spin^c$ Dirac operator on $G$.
	\end{enumerate} 
\end{examples}
Now, let us see how we can define K-theory classes by means of these operators. 
First recall that $D_G^E$ is a regular unbounded self-adjoint operator on $\mathcal{E}(G)$, see \cite{vas}.
Let $\psi$ be a continuous function on the spectrum of $D_G^E$ (which is a subset of $\RR$). We say that it is a \emph{normalizing function} if it is odd (i.e. $\psi(-s)=-\psi(s)\,\,\forall s\in\RR$) and $\lim_{s\to\pm\infty}\psi(s)=\pm1$.
It is a standard fact that the continuous functional calculus of $D_G^E$ by means of $\psi$ gives a continuous operator $\psi(D_G^E)\in \BB(\mathcal{E}(G))$, in particular it is an elliptic 0-order pseudodifferential operator.

\begin{definition}\label{def-ind-class}
Let us denote by $[D_G^E]$ the class in $KK_n(\CC,C^*(G))$ induced by the Kasparov bimodule
$(\mathcal{E}(G),\psi(D_G^E) )$.
\end{definition} 
\begin{remark}\label{psi-path}
	Let $\psi_1$ and $\psi_2$ two normalizing functions for $D_G^E$. Then, for $t\in[0,1]$, $\psi_t:=t\cdot\psi_1+(1-t)\cdot\psi_2$  is a path of normalizing functions. This implies that the class $[D_G^E]$ does not depend on the choice of $\psi$.
\end{remark}

Let $X$ be a closed $G$-invariant smooth submanifold of $M$. Then, as in Remark \ref{fullvsred}, we have a restriction element $[\mathrm{ev}_X]\in KK(C^*(G), C^*(G_{|X}))$. It is immediate to see that 
\begin{equation}\label{restriction-dirac}
[D_G^E]\otimes_{C^*(G)}[\mathrm{ev}_X]=[D_{G_{|X}}^{E_{|X}}]\in KK_n(\CC,C^*(G_{|X})). 
\end{equation}
\emph{From now on, let $E$ be implicitly understood}

\medskip

\noindent\textbf{Fundamental classes.}
	If we see $\mathfrak{A}G$ as a Lie groupoid over $M$, then the corresponding operator $D_{\mathfrak{A}G}$, constructed by using the recipe of Definition \ref{diracgr}, is nothing but the fiber-wise Fourier transform of the Clifford multiplication $cl_G$. It then defines a class 
	\begin{equation}
	\left[\widehat{cl}_G\right]:=\left[\mathcal{E}(\mathfrak{A}G),\psi(\widehat{cl}_G)\right]\in KK_n(\CC, C^*(\mathfrak{A}G)).
	\end{equation}
Let $G$ be a Lie groupoid over a closed smooth manifold $M$, such that its Lie algebroid $\mathfrak{A}G$ is an orientable ($Spin$ or $Spin^c$) $M$-vector bundle, where $M\rightrightarrows M$ denotes the trivial groupoid.
 Recall that $u\colon M \to G$ denotes the unit map and that $\mathcal{N}_u$ is isomorphic to $\mathfrak{A}G$. 
	Then it is easy to check the following equality
	 \begin{equation}\label{symbol-beta}
	[pt]\otimes_{C(M)}\beta(u)=\left[\widehat{cl}_G\right]\in KK_n(\CC, C^*(\mathfrak{A}G))
	\end{equation}
	where $[pt]\in KK(\CC,C(M))$ is the class induced by the collapsing map $pt\colon M\to *$, which is proper if $M$ is compact. 	
	Recall that $\mathfrak{A}G$ is the restriction of $G_{ad}\rightrightarrows M\times[0,1]$ to $M\times\{0\}$ and, as explained in Subsection \ref{DNC-boundary}, this restriction induces a KK-equivalence. So, thanks to \eqref{restriction-dirac}, it is clear that 
	\begin{equation}\label{ad-class}
	[D_{G_{ad}}]=\left[\widehat{cl}_G\right]\otimes_{C^*(\mathfrak{A}G)} [\mathrm{ev}_0]^{-1}\in KK_n(\CC,C^*(G_{ad})).
	\end{equation}	
	 
	\medskip
	
	\noindent\textbf{Primary invariants.}  Again by using
	\eqref{restriction-dirac}, we obtain that
	\begin{equation}\label{ind-class}
	[D_G]=[D_{G_{ad}}]\otimes_{C^*(G_{ad})}[\mathrm{ev}_1]\in KK_n(\CC,C^*(G))
	\end{equation}	
	which is primary invariant or the \emph{index class} of $D_G$.

\medskip
	
	\noindent\textbf{Secondary invariants.} Let us assume that there exists a bounded operator $A$ on $\mathcal{E}(G)$  such that $D_G+A$ is invertible. This implies that $\mathrm{sgn}$, the sign function, is a continuous normalizing function on the spectrum of $D_G+A$. 
	It follows that $(\mathcal{E}(G),\mathrm{sgn}(D_G+A))$ is a degenerate Kasparov bimodule.
	
	\begin{remark}
	In this case it obviously follows that the index class \eqref{ind-class} associated with $D$ is the trivial element in $KK_n(\CC,C^*(G))$.
	\end{remark}
	 Now, we need the following ingredients: let  $\psi$ be any continuous normalizing function for $D_{G_{ad}}$; let $\psi_t=t\cdot\mathrm{sgn}+(1-t)\cdot\psi$; finally, let $\lambda\colon [0,1]\to [0,1]$ be the function which is given by $s\mapsto 2s$ for $s\in [0,1/2]$ and $s\mapsto1$ for $s\in [1/2,1]$.

	 Consider the following two Kasparov bimodule
	 \begin{itemize}
	 	\item $\left(\mathcal{E}(G_{ad}),\psi(D_{G_{ad}})\right)\in \mathbb{E}_n(\CC, C^*(G_{ad}))$, 
	 	\item $\left(\mathcal{E}(G)\otimes C_0[0,1],\psi_t(D_G+\lambda(t)A)\right)\in\mathbb{E}(\CC, C^*(G\times[0,1]))$,
	 \end{itemize}
	 and notice that the evaluation at 1 of  the first one is equal to the evaluation at 0 of the second one. 
	 Let us define, as in \eqref{mc-class}, the following element in in $\mathbb{E}(\CC, C^*(G_{ad}^{[0,1)}))$ by the pair  of Kasparov bimodules \begin{equation}\label{rho-bimodule}\left(\left(\mathcal{E}(G_{ad}),\psi(D_{G_{ad}})\right),\left(\mathcal{E}(G)\otimes C_0[0,1],\psi_t(D_G+\lambda(t)A)\right)\right).\end{equation}
	
	  \begin{definition}\label{rho-def}\cite{zenobi-ad}
	 We will call the class of \eqref{rho-bimodule} the $\varrho$-class of the invertible perturbation of $D_G$ associated with  $A$ by $$\varrho(D_G,A)\in KK_n(\CC, C^*(G_{ad}^{[0,1)})).$$  
	 	If $D_G$ is already invertible, namely $A=0$, we will denote the associated $\varrho$-class just by $\varrho(D_G)$.
	 	
	 \end{definition}

\section{Transverse classes}\label{s4}

\subsection{The isometric case}
Let $\iota\colon H\looparrowright G$ be an immersion of Lie groupoids.
 \emph{Assume that the normal groupoid $\mathcal{N}_\iota$ is the total space of a vector bundle over $H$ of rank $n$ and it is associated with a cocycle $i\colon H\to \mathcal{H}$, where $\mathcal{H}$ is the structural Lie group of $\mathcal{N}_\iota$. Moreover suppose that  $\mathcal{N}_\iota$ is the pull-back through $r$ of an $H$-vector bundle  over $H^{(0)}$ (as for instance in Remark \ref{HS-normal}).}
Let $S\to G$ be a $\mathrm{Cliff}(\mathcal{N}_\iota)$-module.  
 \begin{definition}
 	The lower shriek class associated with $\iota$ is given by the element
 	\begin{equation}
 	 \iota_!(S):= \beta(\iota,S)\otimes\partial(\iota)
 	\end{equation}
 	which belongs to $E_n(C^*(H), C^*(G))$. Here $\beta(\iota,S)$ is as in Remark \ref{HS-normal} and $\partial(\iota)$ as in Definition \ref{E-dnc}.
 \end{definition}

The immersion $\iota$, as in Remark \ref{imm-sub}, induces the following maps:
the immersion of adiabatic groupoids $\iota_{ad}\colon H_{ad}\looparrowright G_{ad}$; by restriction to $[0,1)$, it also induces the immersion $\iota_{ad}^{[0,1)}\colon H^{[0,1)}_{ad}\looparrowright G_{ad}^{[0,1)}$;
finally, by restriction to $\{0\}$, the immersion $d\iota\colon \mathfrak{A}H\looparrowright\mathfrak{A}G$.
These immersions define in turn the lower shriek classes $$(\iota_{ad})_!\,,\quad (\iota_{ad}^{[0,1)})_!\quad \mbox{and} \quad d\iota_!$$ as elements of suitable  E-theory groups. 

\begin{remark}\label{KK-mc}
Notice that, by following the recipe of \cite[Lemma 5.2]{An-Sk} in the KK-theory of mapping cone C*-algebras, is obtained by gluing $(\mathcal{E}_{ad}, F_{ad})$, which represents
$(\iota_{ad})_!$, and $(\mathcal{E}\otimes C_0([0,1)), F\otimes 1)$, where $(\mathcal{E}, F)$ represents $\iota_!$. Here it happens in E-theory.
\end{remark}

It is then easy to show that the following diagram
\begin{equation}\label{diagram!}
\xymatrix{\cdots\ar[r]&K_*(C^*(H)\otimes C_0(0,1))\ar[r]\ar[d]^{\iota_!\otimes id_{C_0(0,1)}}&K_{*}(C^*(H_{ad}^{[0,1)}))\ar[r]\ar[d]^{(\iota_{ad}^{[0,1)})_!}&K_*(C^*(\mathfrak{A}H))\ar[r]\ar[d]^{d\iota_!}&\cdots\\
\cdots\ar[r]&K_{*+n}(C^*(G)\otimes C_0(0,1))\ar[r]&K_{*+n}(C^*(G_{ad}^{[0,1)}))\ar[r]&K_{*+n}(C^*(\mathfrak{A}G))\ar[r]&\cdots}
\end{equation} 
 exists and commutes.

\subsection{The almost isometric case}\label{ai}

Even if the constructions in this section will not be used for the applications, it is still important to present them in order to establish a solid direction for a future generalization of the  product formulas proved in this article. 

Let us consider two foliation $\mathcal{F}_1\subset \mathcal{F}_2$ of rank $n_1$ and $n_2$, respectively, over $M$. Let us denote by $G_1$ and $G_2$ the monodromy groupoids associated with $\mathcal{F}_1$ and $ \mathcal{F}_2$,  respectively. 
The inclusion of the two foliations induces an immersion $\iota\colon G_1\to G_2$ of Lie groupoids. Let us denote by $d\iota$ the inclusion of Lie algebroids $\mathcal{F}_1\hookrightarrow \mathcal{F}_2$ . 
Let us consider the normal groupoid $\mathcal{N}_\iota$, it is isomorphic to $G_1^{\mathcal{F}_2/\mathcal{F}_1}$. Recall that  $\mathcal{F}_2/\mathcal{F}_1$
is a $G_1$-vector bundle over $M$.

\begin{definition}{\cite[Remark 4.3]{HSk}}
	We say that  $\mathcal{F}_2/\mathcal{F}_1$ is almost isometric if there exists a $G_1$-invariant subbundle $E$, of rank $k$, and a splitting $\mathcal{F}_2/\mathcal{F}_1\cong E\oplus E'$ such that both $E$ and $E'$ are endowed with a $G_1$-equivariant isometric structure. This means that the structural group of $\mathcal{F}_2/\mathcal{F}_1$ reduces to the group 
	\begin{equation}\label{aigroup}
	\mathcal{H}:=\left\{\begin{pmatrix}O_k & 0 \\ M_{k, k'}& O_{k'}\end{pmatrix}\right\}
	\end{equation}
	where $k$ and $k'$ are the ranks of $E$ and $E'$ respectively.
	We say that it is almost isometric in the generalized sense if there exists a sequence of $G_1$-invariant vector sub-bundles
	\[
	\mathcal{F}_2/\mathcal{F}_1= E_i\supseteq E_{i-1}\supseteq\dots\supseteq E_1\supseteq E_0=\{0\}
	\]
	with  $E_j/E_{j-1}$ endowed with a $G_1$-equivariant isometric structure  for $j= 1,2,\dots,i$. Namely the structural group of $\mathcal{F}_2/\mathcal{F}_1$ reduces to a group $\mathcal{H}$ of triangular block matrices analogous to \eqref{aigroup}.
	
\end{definition}

Let us assume that 
$\xymatrix{0\ar[r]&E'\ar[r]&\mathcal{F}_2/\mathcal{F}_1\ar[r]&E\ar[r]&0}$
 induces an almost isometric structure on the $G_1$-bundle $\mathcal{F}_2/\mathcal{F}_1$,  namely it is associated with a cocycle $i\colon G_1\to \mathcal{H}$ as in \eqref{aigroup}.
Moreover, let us assume that $\mathcal{F}_2/\mathcal{F}_1$ is $Spin$, hence we have a hermitian $G_1$-vector bundle $S$ of dimension $2^{\lfloor(n_2-n_1)/2\rfloor}$ and a $G_1$-equivariant homomorphism $cl\colon E^*\oplus (E')^*\to End(S)$ such that $cl(\xi)=cl(\xi)^*$, $cl(\xi)^2= \|\xi\|^2$ for all $\xi\in E^*\oplus (E')^*$.

Let $p\colon (\mathcal{F}_2/\mathcal{F}_1)^*\to (E')^* $ be the transposed of the inclusion $E'\hookrightarrow \mathcal{F}_2/\mathcal{F}_1$. Let $s\colon E\to \mathcal{F}_2/\mathcal{F}_1$ any section and let $q\colon (\mathcal{F}_2/\mathcal{F}_1)^*\to E^*$ be its transposed. Finally let $1/2<\rho<1$.

Let us put $b_q(x, \xi):=((1+\|p(\xi)\|^2+\|q(\xi)\|^2)q(\xi),p(\xi))\in (E^*\oplus (E')^*)_x$ for $(x,\xi)\in (\mathcal{F}_2/\mathcal{F}_1)^*_x$ and let us define a bounded multiplier of $C_0((\mathcal{F}_2/\mathcal{F}_1)^*; \pi^*S)$ as follows
\begin{equation}
a_q\colon (x,\xi)\mapsto (1+\|b_q(x,\xi)\|^2)^{-1/2}cl(b_q(x,\xi))\in End(S_x).
\end{equation}

Let us fix a point $x\in M$: we have a vector space  $V:=(\mathcal{F}_2/\mathcal{F}_1)^*_x$ and a representation $E:=S_x$ of $Cliff(V)$.
We naturally obtain 
 an element  $[C_0(V, E), a_q(x,\cdot)]$ in $KK^{\mathcal{H}}_{k}(\CC, C_0(V))$, where  $\mathcal{H}$ is as in \eqref{aigroup}.Thanks to \cite[Lemma 4.1]{HSk}, if we apply to this KK-element the construction of Definition \ref{bott3}, we obtain the element \begin{equation}\label{bott-a.i.}
\beta^{a.i.}(\iota)\in KK_k(C^*(G_1), C^*(\mathcal{N}_\iota)).
\end{equation} 

\begin{definition}
	The lower shriek class associated with an immersion of foliations $\iota\colon \mathcal{F}_1\to\mathcal{F}_2$ with almost isometric normal bundle is given by the element
	\begin{equation}
	\iota_!^{a.i.}:=\beta^{a.i.}(\iota)\otimes \partial(\iota)\in KK_k(C^*(G_1), C^*(G_2)). 
    \end{equation}
\end{definition}
It is obvious from the construction that if the normal bundle of $\iota$ has an isometric structure $\iota_!$ and $\iota^{a.i.}_!$ coincide. 
As in the previous section, it is now easy to define the lower shriek class associated with $\iota_{ad}^{[0,1)}$ and $d\iota$ and obtain diagram \eqref{diagram!} also in the almost isometric case. 

\begin{remark}
	It is immediate from the definition that this is an equivalent description of the class constructed in \cite[Definition 4.2]{HSk}. 
\end{remark}

\section{The equivariant setting}\label{s5}

\subsection{Semidirect products and imprimitivity bimodules}

 Let us first recall some abstract definitions from \cite[Section 3.2]{Kasparov-equiv}.

\begin{definition}\label{definition-equivariant}
Let $X$ be a proper, $\sigma$-compact $\Gamma$-space. Denote by $\{v_i\}$ some countable
approximate unit in $C_0(X/\Gamma)$. For any $\Gamma-C_0(X)$-algebra $B$, define $B^{\Gamma}$ as the
subalgebra of $\mathbb{B}(B)$ consisting of those $\Gamma$-invariant elements $b\in \mathbb{B}(B)$ for which
$f\cdot b\in B$, $\forall f \in C_0(X)$ , and $\lim\|v_ib-b\|=0$ (where elements $v_i$ are considered as
functions on $X$ via the projection $X\to X/\Gamma$). Clearly, $B^\Gamma$ is a $C_0(X/\Gamma)$-algebra.
Now, for any Hilbert $\Gamma-B$-module $\mathcal{E}$, one can define a Hilbert $B^\Gamma$-module
$\mathcal{E}^\Gamma$ as follows. The space of operators $\mathbb{B}(B,\mathcal{E})$ is a Hilbert $\mathbb{B}(B)$-
module. There is an  inclusion $\mathcal{E}\hookrightarrow\mathbb{B}(B,\mathcal{E})$ given by $e\mapsto \tilde{e}$, where $\tilde{e}(b)=e\cdot b$,
$\forall b\in B$. We shall define $\mathcal{E}^\Gamma$ as the subspace of $\mathbb{B}(B,\mathcal{E})$ consisting of those $\Gamma$-invariant
elements $\tilde{e}\in \mathbb{B}(B,\mathcal{E})$ for which $f\cdot \tilde{e}\in \mathcal{E}$, $\forall f\in C_0(X)$, and $\lim\|v_i\tilde{e}-\tilde{e}\|=0$.
\end{definition}

\bigskip

Let $\Gamma$ be a group and let $G\rightrightarrows G^{(0)}$ be a Lie groupoid. Consider an homomorphism $\omega\colon \Gamma\to \mathrm{Aut}(G)$. This induces an action $G\times\Gamma\to G$ given by $(g,\gamma)\mapsto \omega_{\gamma^{-1}}(g)$.

\begin{definition}We can define a groupoid structure on $G\times\Gamma$ over $G^{(0)}$, which we will denote by $G\rtimes_{\omega}\Gamma$, in the following way:
	\begin{itemize}
		\item the source and range maps are given by $s(g,\gamma):=\omega_{\gamma^{-1}}(s(g))$ and $r(g,\gamma):=r(g)$;
		\item the inverse is given by $(g,\gamma)^{-1}:=(\omega_{\gamma}(g^{-1}),\gamma^{-1})$;
		\item the product is defined as follows: $(g,\gamma)\cdot(g',\gamma'):=(g\cdot\omega_{\gamma}(g'),\gamma\gamma')$.
	\end{itemize}
\end{definition}
\begin{proposition}
	Let $\Gamma$ act on $G$ in a free and properly continuous way. Then $G\rtimes\Gamma\rightrightarrows G^{(0)}$ is Morita equivalent to $G/\Gamma\rightrightarrows G^{(0)}/\Gamma$.
\end{proposition}
\begin{proof}
	Let $q\colon G\to G/\Gamma$ denote the quotient map.
	Consider the following groupoid
	$$L:= G\rtimes\Gamma\sqcup \overline{G}\sqcup \overline{G}^{-1}\sqcup G/\Gamma\rightrightarrows G^{(0)}\sqcup G^{(0)}/\Gamma$$
	where as sets $G=\overline{G}=\overline{G}^{-1}$ and
	\begin{itemize}
		\item $s_L$ is equal to: the corresponding source maps on $G\rtimes\Gamma$ and $G/\Gamma$, $q\circ s_G$ on $\overline {G}$ and $s_G$ on $\overline {G}^{-1}$;
		\item  $r_L$ is equal to: the corresponding range maps on $G\rtimes\Gamma$ and $G/\Gamma$, $r_G$ on $\overline {G}$ and $q\circ r_G$ on $\overline {G}^{-1}$;
		\item $i_L$ is equal to: the corresponding inverse maps on $G\rtimes\Gamma$ and $G/\Gamma$, sends $g\in \overline {G}$  to $g^{-1}\in \overline {G}^{-1}$ and similarly for elements in $\overline {G}^{-1}$;
		\item the composition given by:
		\begin{itemize} 
			\item the corresponding one on $G\rtimes\Gamma$ and $G/\Gamma$; 
			\item $(g,\gamma)\cdot g':= (g\cdot\omega_\gamma(g'),\gamma)$ for $(g,\gamma)\in G\rtimes\Gamma$ and $g'\in \overline{G}$;
			\item  $g'\cdot(g,\gamma):=(g'\cdot g,\gamma)$ for $g'\in \overline {G}^{-1}$ and $(g,\gamma)\in G\rtimes\Gamma$; $g\cdot q(g'):=g\cdot \omega_{\gamma'}(g')$ for $g\in \overline{G}$ and $q(g')\in G/\Gamma$, where $\gamma'$ is uniquely determined so that $s(g)=r(\omega_{\gamma}(g'))$;
			\item			 similarly $q(g)\cdot g'= \omega_{\gamma}(g)\cdot g'$ for $q(g)\in G/\Gamma$ and $g'\in \overline {G}$;
			\item for  $g\in\overline{G}$  and $g'\in \overline{G}^{-1}$ with $s(g)=\omega_{\gamma'}(r(g'))$,  $g\cdot g'= g\cdot\omega_{\gamma'}(g')$; 
			\item finally for $g\in\overline{G}^{-1}$  and $g'\in \overline{G}$, we set $g\cdot g'=q(g\cdot g')$.
		\end{itemize}
	\end{itemize}
	By \cite[Definition 2]{Debord-Lescure} the result follows.
\end{proof}
The proof of this classical result is useful here to recall the explicit construction of a link algebra and then the imprimitivity bimodules associated with the Morita equivalence between $C^*(G\rtimes\Gamma)$ and $C^*(G/\Gamma)$. Indeed, $C^*(L)$ does the role of the link algebra and the imprimitivity bimodule is given by the completion of $C^\infty_c(\overline{G})$ inside $C^*(L)$. 

\begin{definition}\label{imprimitivity}
	Let us denote the imprimitivity $C^*(G\rtimes\Gamma),C^*(G/\Gamma)$-bimodule associated with the previous Morita equivalence by $\mathcal{M}_\Gamma^G$.
\end{definition}

	Let $E\to G^{(0)}$ be a $\Gamma$-equivariant vector bundle. Denote by $\mathcal{E}_c$ the $C_c^\infty(G)$-module $C^\infty_c(G, r^*E\otimes \Omega^{\frac{1}{2}})$ and by $\mathcal{E}$ its C*-completion with respect to the usual $C^*(G)$-inner product 
	\begin{equation}\langle\xi,\xi'\rangle(\gamma):=\int_{G_{s(\gamma)}}\langle\overline{\xi(\gamma\eta^{-1})}\xi'(\eta)\rangle_E,\end{equation}
	 moreover recall that the right action of $C^*(G)$ on $\mathcal{E}$ is given by \begin{equation}\label{rightpr}\xi\cdot f(\gamma):=\int_{G_{s(\gamma)}}\xi(\gamma\eta^{-1})f(\eta).\end{equation}
	As in \cite[Definition 3.8]{Kasparov-equiv}, associated with the action of $\Gamma$ on $\mathcal{E}$, we can construct the $C^*(G)\rtimes\Gamma$-module $\mathcal{E}\rtimes\Gamma$ as the C*-completion of $C_c\left(\Gamma,C^\infty_c(G, r^*E\otimes \Omega^{\frac{1}{2}})\right)$.
	
    Another way to obtain a $C^*(G)\rtimes\Gamma$-module from   $\mathcal{E}_c$ is to endow it with a $C^*(G)\rtimes\Gamma$-valued inner product $\langle\cdot,\cdot\rangle_{\rtimes}$ defined by $\langle\xi,\eta\rangle_{\rtimes}(\gamma):=\langle\xi,\omega_\gamma^*(\eta)\rangle$. Let us denote its C*-completion by $\mathcal{E}^{\rtimes}$.

    Furthermore, observe that $\mathcal{E}^\Gamma$, the $C^*(G/\Gamma)$-module of $\Gamma$-invariant elements of $\mathcal{E}$, is isomorphic to the $C^*(G/\Gamma)$-completion of $C^\infty_c(G/\Gamma, r^*\bar{E}\otimes \Omega^{\frac{1}{2}})$, where $\bar{E}$ is the quotient of $E$ by the action of $\Gamma$.
    
Finally, notice that both $\mathcal{E}^\Gamma$ and $\mathcal{E}^\rtimes$ are left $C(G^{(0)}/\Gamma)$-modules, that $\mathcal{E}$ is a  left $C(G^{(0)})$-module and then that $\mathcal{E}\rtimes\Gamma$ is a left $C_0(G^{(0)})\rtimes\Gamma$-module.

    \begin{lemma}\label{iso-modules}
    	There exists an isomorphism   between  $\mathcal{E}^\rtimes$ and $\mathcal{E}^\Gamma\otimes_{C^*(G/\Gamma)}(\mathcal{M}_\Gamma^{G})^{-1}$ as $C^*(G)\rtimes\Gamma$-modules.
    \end{lemma}
\begin{proof}
	Consider the morphism $\alpha\colon \mathcal{E}^\Gamma_c\otimes_{C_c^{\infty}(G/\Gamma)}C_c^\infty(G)\to \mathcal{E}_c$ given by $\alpha(\tilde{\eta}\otimes\xi)=\tilde{\eta}(\xi)$. Following the reasoning in \cite[Lemma 3.2]{Kasparov-equiv} - notice that here we consider an isomorphism of $C^*(G)\rtimes\Gamma$-modules whereas there it is considered a morphism of $C^*(G)$-modules -  it is easy to prove that this is an isometric map with dense range. 
\end{proof}

\begin{remark}
	Observe that from the proof of \cite[Theorem 3.14]{Kasparov-equiv}, we have the following isomorphisms of $C_0(G^{(0)}),C^*(G)\rtimes\Gamma$-bimodules
	\begin{equation}
	\mathcal{E}^\rtimes\cong\mathcal{E}^\Gamma\otimes_{C^*(G/\Gamma)} (\mathcal{M}_\Gamma^{G})^{-1}\cong (\mathcal{M}_\Gamma^{G^{(0)}})^{-1}\otimes_{C_0(G^{(0)})\rtimes\Gamma} \mathcal{E}\rtimes\Gamma
	\end{equation}
\end{remark}
%
%
%
%

\subsection{Equivariant KK-classes}

Consider the following situation: let $\Gamma$ be a discrete group acting on the Lie groupoid $G$ in a free and properly discontinuous way; let $i\colon G\to  \mathcal{H}$ be a $\Gamma$-invariant cocycle, namely it descends to a cocycle $l\colon G/\Gamma\to \mathcal{H}$.  Then we have the following result. 

\begin{proposition}\label{eq-bott}
	The cocycle $i$ defines a map $i\colon KK^{\mathcal{H}}(A,B)\to KK^\Gamma(A\rtimes_i G,B\rtimes_i G)$ and the following equality holds
	\begin{equation}\label{equiv-cocycle}
	\mathcal{M}_{A,\Gamma}^{-1}\otimes j^\Gamma(i^*(x))\otimes\mathcal{M}_{B,\Gamma}= l^*(x)
	\end{equation}
	where $x\in KK^{\mathcal{H}}(A,B)$ and $\mathcal{M}_{A,\Gamma}$ is an imprimitivity bimodule for the algebras $(A\rtimes_i G)\rtimes\Gamma$ and $A\rtimes_l(G/\Gamma)$ and similarly for $\mathcal{M}_{B,\Gamma}$.
\end{proposition}
\begin{proof}
	The first statement follows from the fact that the KK-element $\mathrm{Id}_{C^*(G^{P_i})}$ actually lies in $KK^\Gamma(C^*(G^{P_i}),C^*(G^{P_i}))$ and \eqref{equiv-cocycle} follows from the following calculations:
	\begin{equation}
	\begin{split}
	&\mathcal{M}_{A,\Gamma}^{-1}\otimes j^\Gamma(i^*(x))\otimes\mathcal{M}_{B,\Gamma}=\\
	=&\mathcal{M}_{A,\Gamma}^{-1}\otimes j^\Gamma(\mathcal{M}_{A,i}\otimes j^{\mathcal{H}}(\mathrm{Id}_{C^*(G^{P_i})}\otimes x)\otimes \mathcal{M}^{-1}_{B,i})\otimes\mathcal{M}_{B,\Gamma}=\\
	=&\mathcal{M}_{A,\Gamma}^{-1}\otimes j^\Gamma(\mathcal{M}_{A,i})\otimes j^\Gamma\circ j^{\mathcal{H}}(\mathrm{Id}_{C^*(G^{P_i})}\otimes x)\otimes j^\Gamma(\mathcal{M}^{-1}_{B,i})\otimes\mathcal{M}_{B,\Gamma}=\\
	=&\mathcal{M}_{A,\Gamma}^{-1}\otimes j^\Gamma(\mathcal{M}_{A,i})\otimes j^{\mathcal{H}}(\mathcal{M}_{\Gamma}^{i,l})\otimes j^{\mathcal{H}}(\mathcal{M}_{\Gamma}^{i,l})^{-1}\otimes j^{\mathcal{H}}\circ j^\Gamma(\mathrm{Id}_{C^*(G^{P_i})}\otimes x)\otimes\\
	&\otimes  j^{\mathcal{H}}(\mathcal{M}_{\Gamma}^{i,l})\otimes j^{\mathcal{H}}(\mathcal{M}_{\Gamma}^{i,l})^{-1}\otimes  j^\Gamma(\mathcal{M}^{-1}_{B,i})\otimes\mathcal{M}_{B,\Gamma}=\\
	=&\mathcal{M}_{A,l}^{-1}\otimes j^\mathcal{H}\left((\mathcal{M}_{\Gamma}^{i,l})^{-1}\otimes j^\Gamma(\mathrm{Id}_{C^*(G^{P_i})}\otimes x)\otimes \mathcal{M}_{\Gamma}^{i,l}\right)\otimes \mathcal{M}_{B,l}=\\
	=&\mathcal{M}_{A,l}^{-1}\otimes j^\mathcal{H}\left(\left((\mathcal{M}_{\Gamma}^{i,l})^{-1}\otimes j^\Gamma(\mathrm{Id}_{C^*(G^{P_i})})\otimes \mathcal{M}_{\Gamma}^{i,l}\right)\otimes x\right)\otimes \mathcal{M}_{B,l}=\\
	=&\mathcal{M}_{A,l}^{-1}\otimes j^\mathcal{H}\left(C^*((G/\Gamma)^{P_l})\otimes x\right)\otimes \mathcal{M}_{B,l}=l^*(x)
	\end{split}
	\end{equation}
	where we used the distibuitivity of the descent maps and the fact that $j^\Gamma$ and $j^\mathcal{H}$ commute since the two actions commute; moreover $\mathcal{M}_\Gamma^{i,l}$ is the imprimitivity bimodule for the C*-algebras $C^*(G^{P_i}\rtimes\Gamma)$ and $C^*((G/\Gamma)^{P_l})$ as in Definition \ref{imprimitivity}; finally we repeatedly used the fact that a crossed product bimodule $\mathcal{E}\rtimes\Gamma$ is conjugated the to bimodule of invariant elements $\mathcal{E}^\Gamma$ by means of interior tensor product with suitable imprimitivity bimodules, see the proof of \cite[Theorem 3.14]{Kasparov-equiv} for the details about this fact. 
\end{proof}

\begin{proposition}\label{eq-partial}
	Let $\Gamma$ be a discrete group acting in a free and properly discontinuous way on both $G$ and $H$. Let $\iota\colon G\looparrowright H$ be a $\Gamma$-equivariant immersion. Then the element $\partial^G_H$ in \eqref{partial} belongs to $E^\Gamma(C^*(\mathcal{N}_H^G), C^*(G))$ and \begin{equation}(\mathcal{M}_\Gamma^{\mathcal{N}})^{-1}\otimes j^\Gamma(\partial^G_H)\otimes \mathcal{M}_\Gamma^G=\partial^{G/\Gamma}_{H/\Gamma}.\end{equation}
\end{proposition}
\begin{proof}
	The proof is immediate, since the element in \eqref{partial} is the composition of a morphism and the inverse of a morphism.
\end{proof}

 From now on let us assume that the action of $\Gamma$ on $G^{(0)}$ is cocompact, so that the map $pt\colon G^{(0)}/\Gamma\to \ast$ is proper. Fix a $\Gamma$-invariant metric $\overline{g}$ on $AG$, which descends to a metric $g$ on $AG/\Gamma$. Thanks to Proposition \ref{eq-bott} and $\ref{eq-partial}$ one can easily see that the following equalities hold:
 \begin{equation}
 [\hat{cl}_{G/\Gamma}]=[pt]\otimes(\mathcal{M}_\Gamma^{G^{(0)}})^{-1}\otimes j^\Gamma(\beta(u))\otimes \mathcal{M}^{AG}_{\Gamma}, 
 \end{equation}
 where $\beta(u)$ belongs to $KK^\Gamma_n(C_0(G^{(0)}), C^*(AG))$;
\begin{equation}\label{lift-dirac}
[D_{G/\Gamma}]= [pt]\otimes(\mathcal{M}_\Gamma^M)^{-1}\otimes j^\Gamma([D_G])\otimes \mathcal{M}^G_{\Gamma},
\end{equation}
where $[D_G]$ is the class of the Dirac operator in $KK^\Gamma_n(C_0(G^{(0)}), C^*(G))$.

\begin{proposition}\label{lift-shriek}
	Let $\Gamma$ be a discrete group acting in a free and properly discontinuous way on both $G$ and $H$. Assume that $\overline{\iota}\colon G\to H$ is $\Gamma$-equivariant and that the normal bundle of $\bar{\iota}$ is associated with a $\Gamma$-invariant cocycle. Then $(\bar{\iota}_{ad}^{[0,1)})_!$  gives an element in $E^\Gamma_n(C^*(H_{ad}^{[0,1)}), C^*(G_{ad}^{[0,1)}))$. Moreover, let $\iota\colon G/\Gamma\to H/\Gamma$ be the immersion induced by $\iota$ between the quotients, then \begin{equation}
	(\mathcal{M}_\Gamma^{H_{ad}})^{-1}\otimes j^\Gamma\left((\overline{\iota}_{ad}^{[0,1)})_!\right)\otimes (\mathcal{M}_\Gamma^{G_{ad}})=(\iota_{ad}^{[0,1)})_!. 
	\end{equation}
\end{proposition}

\begin{proof}
	It follows immediately from Proposition \ref{eq-bott} and Proposition \ref{eq-partial}.
\end{proof}

Now, let us consider a Dirac type operator $D_{G/\Gamma}$ acting on a $C^*(G/\Gamma)$-module $\mathcal{E}^\Gamma$, given by the completion of  $\overline{C^{\infty}_c(G/\Gamma, r^*\overline{E}\otimes\Omega^{\frac{1}{2}})}$, and let us assume that $D_{G/\Gamma}$ is invertible.
The operator $D_{G/\Gamma}\otimes 1$ on $\mathcal{E}^\Gamma\otimes(\mathcal{M}_\Gamma^{G})^{-1}$ defines, by means of the isomorphism $\alpha$ in Lemma \ref{iso-modules} an operator on $\mathcal{E}^{\rtimes}$ which we can identify in the following way. Take $\eta\in \overline{C^{\infty}_c(G/\Gamma, r^*\overline{E}\otimes\Omega^{\frac{1}{2}})}$ and $\xi\in C^\infty_c(G,\Omega^{\frac{1}{2}})$, then we have 
\begin{equation}\label{dirac-cross}
\begin{split}
&\alpha\circ D_{G/\Gamma}\otimes 1\circ\alpha^{-1}(\tilde{\eta}(\xi))=\\
=& \alpha\circ D_{G/\Gamma}\otimes 1(\tilde{\eta}\otimes \xi)=\\
=& \alpha(D_{G/\Gamma}(\tilde{\eta})\otimes \xi)=D_{G/\Gamma}(\tilde{\eta})(\xi).
\end{split}
\end{equation}
The last term, following Definition \ref{definition-equivariant}, corresponds to the $\Gamma$-equivariant lift $D_G^\Gamma$ of $D_{G/\Gamma}$ acting on the element $\eta\cdot \xi$ defined as in \eqref{rightpr}, which is an element  in $ C^{\infty}_c(G, r^*E\otimes\Omega^{\frac{1}{2}})$ seen as a $C^*(G)\rtimes\Gamma$-module. Here $E$ is the pull-back of $\overline{E}$ through the quotient map $G^{(0)}\to G^{(0)}/\Gamma$. Namely $D^\Gamma_G$ defines an operator on the Hilbert module $\mathcal{E}^\rtimes$. Finally observe that $D^{\Gamma}_G$ is invertible. 

\begin{proposition}\label{lift-rho}
	We have that $\varrho(D_{G/\Gamma})\otimes (\mathcal{M}^G_\Gamma)^{-1}= \varrho(D^\Gamma_G)\in KK(\CC,C^*(G)\rtimes \Gamma)$. 
\end{proposition}

\begin{proof}
	Let the subscript $ad$ indicate the Hilbert module associated with the pull-back of the vector bundle $E$ on the adiabatic deformation groupoid. 
	Then, let us recall that $\varrho(D_{G/\Gamma})$ is represented, as in \eqref{mc-class}, by 
	\begin{equation}
   \left[  \left(pt^*\mathcal{E}^\Gamma_{ad}, \psi(D_{G/\Gamma_{ad}})\right),\left(pt^*\mathcal{E}^\Gamma\otimes C_0[0,1], \psi_t(D_{G/\Gamma})\right)\right] \in KK(\CC, C^*(G/\Gamma_{ad}^{[0,1)}))
	\end{equation}
	as in \eqref{rho-bimodule}. 
	Analogously $\varrho(D^\Gamma_G)$ is represented by 
	\begin{equation}\label{rho-cross}
	\left[ \left(pt^*\mathcal{E}^\rtimes_{ad}, \psi(D_{G_{ad}})^\Gamma\right),\left(pt^*\mathcal{E}^\rtimes\otimes C_0[0,1], \psi_t(D^\Gamma_{G})\right)\right] \in KK(\CC, C^*(G_{ad}^{[0,1)})\rtimes\Gamma).
	\end{equation}
	Now, consider $\varrho(D_{G/\Gamma})\otimes_{C^*(G/\Gamma)}(\mathcal{M}^G_\Gamma)^{-1}$ and let us represent it by following the recipe given in \eqref{k-product-mc}. Thus, thanks to Lemma \ref{iso-modules}, we obtain the mapping cone class given, as in \eqref{mc-class}, by the pair
		\begin{equation}\begin{split}
	 &\left(pt^*\mathcal{E}^\rtimes_{ad}, \psi(D_{G/\Gamma_{ad}})\otimes 1\right) \,\mbox{and}\\
	 \left(pt^*\mathcal{E}^\rtimes\otimes C_0[0,1], \psi_t(D_{G/\Gamma})\otimes 1\right)
&	\diamond (pt^*\mathcal{E}^\rtimes\otimes C_0[0,1], CS_t(\mathrm{sgn}(D_{G/\Gamma})\otimes 1, \mathrm{sgn}(D_{G/\Gamma}\otimes 1)) 
		\end{split}
		\end{equation}
    which, thanks to \eqref{dirac-cross} and the fact that $\mathrm{sgn}(D_{G/\Gamma})\otimes 1= \mathrm{sgn}(D_{G/\Gamma}\otimes 1)$, is equal to \eqref{rho-cross}. 
\end{proof}

\section{Product formulas}\label{s6}

\subsection{Smooth fibrations}

Let $(M,g^M)$  and $(B,g^B) $ be two smooth compact connected  Riemannian manifolds and let $\pi\colon M\to B$ be a surjective Riemannian submersion.
This means that the tangent bundle $TM$ splits as $\ker d\pi\oplus \pi^*TB$ and that the metric $g^M$ can be expressed as the sum $g^{M/B}\oplus\pi^*g^B$, where $g^{M/B}$ is a metric for $\ker d\pi$.

Let $\nabla^\alpha$ be the Levi-Civita connection for $g^\alpha$, with $\alpha= M, B$ or $M/B$. We thus obtain two metric connections $\nabla^M $ and $\nabla^\oplus:=\nabla^{M/B}\oplus\pi^*\nabla^B$ on $TM$
whose difference is calculated in terms of the tensor
$\omega\in C^\infty(M;T^*M\otimes\bigwedge^2T^*M)$ defined by 
\begin{equation}
\begin{split}
\omega(X)(Y,Z):= &S(X,Z,Y)- S(X,Y, Z)+\\
+&\frac{1}{2}\left(\Omega(X,Z,Y)-\Omega(X,Y,Z)+\Omega(Y,Z,X)\right)
\end{split}
\end{equation}
for $X,Y,Z\in C^\infty(M;TM)$. Here, set $P$ the projection $TM\to \ker d\pi$, then
\begin{itemize}
	\item $S\in C^\infty(M;T^*M^{\otimes3})$ is the \emph{second fundamental form} defined by
\begin{equation}
S(X,Y,Z):=g^M\left(\nabla^{M/B}_{(1-P)Z}(PX)-[(1-P)Z,PX], PY\right),
	\end{equation}
	\item we will also need the trace of $S$, which gives a 1-form $k\in C^\infty(M;T^*M)$, the \emph{mean curvature}, defined as
	\begin{equation}
	k(X):=\sum_i S(e_i, e_i, X),
    \end{equation}
	\item $\Omega\in C^\infty(M; \bigwedge^2T^*M\otimes T^*M)$ is the \emph{curvature of the fibration $\pi$} defined by 
\begin{equation}
\Omega(X,Y,Z):= -g^M\left([(1-P)X,(1-P)Y],PZ\right),
	\end{equation}
\end{itemize}
for $X,Y,Z\in C^\infty(M;TM)$.

\begin{lemma}
	The connections $\nabla^M$ and $\nabla^\oplus$ are related by the following formula
	\begin{equation}\label{sumformula}
	g^M\left(\nabla^M_X Y, Z\right)= g^M\left(\nabla^\oplus_X Y, Z\right)+\omega(X)(Y,Z).
	\end{equation}
\end{lemma}

In \cite{BismutCheeger} and \cite{KaadVS} the following factorization result is stated for the $Spin^{(c)}$ Dirac operator, but the proof, which essentially depends only on \eqref{sumformula}, works verbatim for generalized Dirac operator too. 
Consider the following objects:
\begin{itemize}
	\item the Dirac operator $D^B$ acting on the sections of a bundle  of $\mathrm{Cliff}(TB, g^{B})$-modules $E^B\to B$;
	\item the  Dirac operator $D^{M/B}$ acting on the sections of a bundle of $\mathrm{Cliff}(\ker d\pi, g^{M/B})$-modules $E^{M/B}\to M$.
\end{itemize}
Let us consider the following metric depending on a parameter $\varepsilon$
\begin{equation}
g^M_\varepsilon= g^{M/B}\oplus \varepsilon^{-1}\pi^*g_B,
\end{equation}
then the Dirac-type operator $D_\varepsilon^M$  acts on the sections of the bundle of $\mathrm{Cliff}(TM,g_\varepsilon^M)$-modules $E^M:=E^{M/B} \otimes\pi^*E^B$ and  can be written, as in \cite[Eq. (4.26)]{BismutCheeger}, in the following way 
\begin{equation}\label{factorization}
D^M_\varepsilon=  D^{M/B}\otimes 1+ \varepsilon^{\frac{1}{2}}\sum_{i}c(f_i)\widetilde{\nabla}^B_{f_i}- \frac{\varepsilon}{4}\sum_{i<j}c(f_if_j)c(\Omega(f_i,f_j)) 
\end{equation}
where $\widetilde{\nabla}^B_X= (\pi^*\nabla^B)_X-\frac{1}{2}k(X)$ for $X\in C^\infty(M;TM)$ and $\{f_i\}$ is a local orthonormal frame of $TM$.

In \cite{KaadVS} it is meticulously proved that 
the unbounded Kasparov $C(M)$-$\CC$-bimodule $(L^2(M;E^M), D^M_{\varepsilon})$ is a bounded perturbation of the unbounded Kasparov product of the unbounded $C(M)$-$C(B)$-bimodule $(L^2(M;E^{M/B}), D^{M/B})$ and the unbounded $C(B)$-$\CC$-bimodule $(L^2(B;E^B), D^B_{\varepsilon})$, where $D^B_{\varepsilon}$ is the Dirac operator associated with the metric $\varepsilon^{-1}g_B$ and where the bounded defect is given by the last term $- \frac{\varepsilon}{4}\sum_{i<j}c(f_if_j)c(\Omega(f_i,f_j))$ in \eqref{factorization}.

Now, notice that the Lie groupoid  $M\times_\pi M\rightrightarrows M$  is Morita equivalent to $B\rightrightarrows B$ and that $M\times M\rightrightarrows M$  is Morita equivalent to the trivial groupoid $pt\rightrightarrows pt$. Let us denote by $\mathcal{M}_\pi$ and $\mathcal{M}_{pt}$ the imprimitivity bimodules associated with these Morita equivalences. More precisely: $\mathcal{M}_\pi$ is the closure of $C_c^\infty(M)$ equipped with the $C(B)$-valued inner product $\langle f,g\rangle(b):=\int_{\pi^{-1}(b)}\overline{f(x)}g(x)$ with the obvious left action of $C_c^\infty(M\times_\pi M)$; $\mathcal{M}_{pt}$ is just $L^2(M)$ with the obvious left action of $C^*(M\times M)$.  Now, it is an easy observation that the previous factorization corresponds, by means of these Morita equivalences,  to the factorization at the unbounded level of the following composition of inclusions of Lie groupoids over $M$
\[
\xymatrix{M\ar[r]_(.35){u^\pi}\ar@/^2pc/[rr]^{u}&M\times_\pi M\ar[r]_{\iota}& M\times M}
\]
namely that $ u_!=u^\pi_!\otimes \iota_!$, where $u^\pi$ is the unit map of $M\times_\pi M $ and $u$ is the unit map of $M\times M$. Thanks to \eqref{symbol-beta}, \eqref{ad-class}, \eqref{ind-class}, by precomposing with the KK-class  $[pt]$ associated with the map $M\to pt$, this just gives us the unbounded factorization of the primary class $[D_{M\times M}]= [D_{M\times_\pi M}]\otimes \iota_!$.

In particular, with the notation of Definition \ref{def-ind-class} and $\simeq$ standing for unitarily equivalent, observe that:
\begin{itemize}
	\item $\left(\mathcal{E}(M\times M, E^M), D^{E^M}_{M\times M}\right)\otimes \mathcal{M}_{pt}\simeq(L^2(M;E^M), D^M)\in \mathbb{E}(\CC,\CC)$;  
	\item $\left(\mathcal{E}(M\times_\pi M, E^{M/B}), D^{E^{M/B}}_{M\times_\pi M}\right)\otimes \mathcal{M}_{\pi}\simeq(L^2(M\to B;E^{M/B}), D^{M/B})\in \mathbb{E}(\CC,C(B))$;
	\item then, up to the last bounded term in \eqref{factorization},
	\begin{equation}
	\begin{split}	
&\left(\mathcal{E}(M\times M, E^M), D^{E^M}_{M\times M}\right)\simeq\\
&\simeq \left(L^2(M;E^M), D^M\right)\otimes \mathcal{M}_{pt}^{-1}\simeq\\
	&\simeq \left(L^2(M\to B;E^{M/B}), D^{M/B}\right)\otimes \left(L^2(B;E^B), D^B\right)\otimes \mathcal{M}_{pt}^{-1}\simeq\\
	&\simeq \left(\left(\mathcal{E}(M\times_\pi M, E^{M/B}), D^{E^{M/B}}_{M\times_\pi M}\right)\otimes \mathcal{M}_{\pi}\otimes(L^2(B;E^B), D^B)\right)\otimes \mathcal{M}_{pt}^{-1}.
	\end{split}
	\end{equation} 
\item so that 
\begin{equation}\label{D-epsilon}D^{E^M}_{M\times M,\varepsilon}= D^{E^{M/B}}_{M\times_\pi M}\otimes1\otimes 1 + \varepsilon^{\frac{1}{2}} (1\otimes_{\nabla}D^B)\otimes1- \varepsilon A\end{equation}
where $1\otimes_{\nabla}D^B$ is a $D^B$-connection on $\left(\mathcal{E}(M\times_\pi M, E^{M/B})\otimes \mathcal{M}_\pi\right)\otimes L^2(B; E^B)$, unitarily equivalent to the second term in \eqref{factorization}, and $A$ is a zero order term corresponding to the last term in \eqref{factorization}.
\end{itemize}

The reason we did this detailed translation to the Lie groupoid setting of the factorization \eqref{factorization} is that we need a fine estimate of the commutator of $D^{E^{M/B}}_{M\times_\pi M}\otimes1\otimes 1$ and $D^{E^M}_{M\times M,\varepsilon}$ in order to apply Lemma \ref{lemma-cs-degenerate} in the following theorem. 

\begin{theorem}\label{product-fibration}
	Let $\pi\colon (M,g^M)\to (B,g^B)$ be a Riemannian submersion between $Spin$ manifolds. Let $g^{M/B}$ be a metric of $\ker d\pi$ with positive scalar curvature and let $\varrho(D_{M\times_\pi M})\in KK_n(\CC, C^*(M\times_\pi M_{ad}^{[0,1)}))$ be the associated $\varrho$-class. Let $\varepsilon_0$ be such that $g^M_{\varepsilon_0}$ has positive scalar curvature too. Then we have that 
	\begin{equation}\label{product-formula-fibrations}
\varrho(D_{M\times_\pi M})\otimes \iota_!^{ad}=\varrho(D^{\varepsilon_0}_{M\times M})\in KK_{n+k}(\CC, C^*(M\times M_{ad}^{[0,1)}))
	\end{equation}
	where $D$ denotes the $Spin$ Dirac operator, $\iota$ is the inclusion of Lie groupoids $M\times_\pi M\hookrightarrow M\times M$, $n=rk(\ker d\pi)$ and $k=\dim B$.
\end{theorem}

\begin{proof}
	Let us denote by $H$ the Lie groupoid $M\times_\pi M$ and by $G$ the Lie groupoid $M\times M$. Recall from \eqref{rho-bimodule} that $$\varrho(D_H)=\left[\left(\mathcal{E}(H_{ad}),\psi(D_{H_{ad}})\right)\left(\mathcal{E}(H)\otimes C_0[0,1],\psi_t(D_H)\right)\right]\in KK_n(\CC, C^*(H_{ad}^{[0,1)})).$$
	
	Moreover, as in \eqref{k-product-mc}, we have that $\varrho(D_{M\times_\pi M})\otimes \iota_!^{ad}\in KK_{n+k}(\CC, C^*(G_{ad}^{[0,1)}))$
	is represented   by 
	\begin{equation}\label{product1}
 	\left[\left(\mathcal{E}(G_{ad}),\psi(D_{G_{ad}})\right),\left(\mathcal{E}(G)\otimes C_0[0,1],CS_t(\psi(D_G),\mathrm{sgn}(D_H)\otimes1 )\right)\right].
	\end{equation}
    Observe that the class of $\iota_!$ does not depend on $\varepsilon$ which parametrizes the family of metrics $\varepsilon^{-1}g_B$ one can use to construct it. Then the family of bimodules
    \begin{equation}\label{product2}
   \left[	\left(\mathcal{E}(G_{ad}),\psi(D^\varepsilon_{G_{ad}})\right)\left(\mathcal{E}(G)\otimes C_0[0,1],CS_t(\psi(D^{\varepsilon}_G),\mathrm{sgn}(D_H)\otimes1 )\right)\right] 
    \end{equation} 
   defines the same class in $KK_{n+k}(\CC, C^*(G_{ad}^{[0,1)}))$ for any $\varepsilon$. 
   Observe now that \cite[Equation (4.46)]{BismutCheeger}, see also \cite[Lemma 17]{KvS2} for more details, tells us that there exists an $\varepsilon_0$ and $\lambda>-2 $ such that
   the commutator $[D_G^{\varepsilon_0}, D_H\otimes 1]-\lambda\geq0$ and that $D_G^{\varepsilon_0}$ is invertible.
   Thanks to Lemma \ref{lemma-cs-degenerate} 
       \begin{equation}
    \left[ \left(\mathcal{E}(G_{ad}),\psi(D^{\varepsilon_0}_{G_{ad}})\right),\left(\mathcal{E}(G)\otimes C_0[0,1],CS_t(\psi(D^{\varepsilon_0}_G),\mathrm{sgn}(CS_s(D_G^{\varepsilon_0},D_H\otimes 1 ))\right)\right]
       \end{equation} 
   gives a homotopy, parametrized by $s$, from \eqref{product2} to $\varrho(D^{\varepsilon_0}_{M\times M})$.	
\end{proof}

Observe now that the Lie groupoid $M\times_B M\rightrightarrows M$ is the holonomy groupoid $\mathrm{Hol}(\pi)$  associated with the foliation $(M,\ker d\pi)$. If the fibers are not simply connected, it differs from the monodromy groupoid $\mathrm{Mon}(\pi)$, whose $s$-fibers are diffeomorphic to the universal covering of the typical fiber of $\pi$. This last groupoid admits an immersion $\tilde{\iota}\colon \mathrm{Mon}(\pi)\looparrowright \widetilde{M}\times_\Gamma \widetilde{M}$ of Lie groupoids over $M$, where $\Gamma=\pi_1(M)$.
Notice that the proof of the previous theorem is obtained from local considerations about differential operators which adapt almost verbatim to the higher situation through equivariant lift. Hence, we can directly state the following more general result. 
\begin{theorem}
	Let $\pi\colon (M,g^M)\to (B,g^B)$ be a Riemannian submersion between $Spin$ manifolds. Let $g^{M/B}$ be a metric of $\ker d\pi$ with positive scalar curvature and let $\varrho(D_{\mathrm{Mon}(\pi)})\in KK_n(\CC, C^*(\mathrm{Mon}(\pi)_{ad}^{[0,1)}))$ be the associated $\varrho$-class. Let $\varepsilon_0$ be such that $g^M_{\varepsilon_0}$ has positive scalar curvature too. Then we have that 
	\begin{equation}\label{product-formula-fibrations-mon}
	\varrho(D_{\mathrm{Mon}(\pi)})\otimes \tilde{\iota}_!^{ad}=\varrho(D^{\varepsilon_0}_{\widetilde{M}\times_\Gamma \widetilde{M}})\in KK_{n+k}(\CC, C^*(\widetilde{M}\times_\Gamma \widetilde{M}_{ad}^{[0,1)})).
	\end{equation}
\end{theorem}

%
%

\subsection{Foliated bundles}
Let $\Gamma$  be a discrete group of isometries acting on a smooth Riemannian manifold $(\bar{M},g^{\bar{M}})$ freely, properly discontinuously and so that $(M,g^M):=(\bar{M},g^{\bar{M}})/\Gamma$ is a compact smooth Riemannian manifold. 
Suppose that we also have an isometric action of $\Gamma$ on a compact Riemannian manifold $(B, g^B)$ and
a surjective Riemannian submersion $\pi\colon \bar{M}\to B$, which is $\Gamma$-equivariant. The simple foliation of $\bar{M}$ associated with the involutive sub-bundle $\ker d\pi\subset T\bar{M}$ is invariant under the action of $\Gamma$ and hence it induces the quotient foliation $\mathcal{F}(\pi, \Gamma):=\ker d\pi/\Gamma$ on $M$. Now, as in the previous section, $g^{\bar{M}}$ is equal to $ g^{\bar{M}/B}\oplus\pi^* g^B$, where 
$g^{\bar{M}/B}$ is a metric on $\ker d\pi$ which induces a metric $g^{\mathcal{F}}$ on the quotient sub-bundle $\mathcal{F}(\pi,\Gamma)\subset TM$.

If $\Gamma_b\subset \Gamma$ is the isotropy group at $b\in B$ of the action of $\Gamma$ on $B$, then the leaf of $(M,\mathcal{F}(\pi, \Gamma))$  obtained from the leaf $\pi^{-1}(b)$ of $(\bar{M}, \ker d\pi)$ is naturally diffeomorphic to $\pi^{-1}(b)/\Gamma_b$.

Observe that $\Gamma$ induces a free and properly discontinuous action via groupoid automorphisms on the Lie groupoids associated with the foliation on $\bar{M}$ and hence we have the following identifications of Lie groupoids over  $M$
\begin{equation*}
\mathrm{Mon}(M,\mathcal{F}(\pi, \Gamma))\cong \mathrm{Mon}(\bar{M},\ker d\pi)/\Gamma\, \,\mbox{and}\, \,\mathrm{Hol}(M,\mathcal{F}(\pi, \Gamma))\cong \mathrm{Hol}(\bar{M},\ker d\pi)/\Gamma
\end{equation*}
for the monodromy and the holonomy groupoid, respectively. See \cite[Example 5.8]{MoeMcr}, for instance. 
Finally, it is worthy to point out that the Lie algebroid of both these groupoids is given by $\ker d\pi/\Gamma\to M$.


Observe now that there is a commutative diagram of Lie groupoid morphisms
\begin{equation}
\xymatrix{ \mathrm{Mon}(\bar{M},\ker d\pi) \ar[r]^{\bar{\iota}}\ar[d]& \mathrm{Mon}(\bar{M},T\bar{M})\ar[d]\\
\mathrm{Mon}(M,\mathcal{F}(\pi, \Gamma))\ar[r]^{\iota}&\mathrm{Mon}(M,TM)}
\end{equation}
where the horizontal arrows are immersions and the vertical ones are quotient maps.

\begin{theorem}\label{main-th}
	In the previous geometric situation, let  the metric $g^{\mathcal{F}}$ on $\ker d\pi/\Gamma$  be a metric with positive scalar curvature. Let us assume that both $M$ and $B$ are endowed with a $\Gamma$-invariant $Spin$ structure. Then there exists an $\varepsilon>0$ such that 
	\begin{equation}\label{product-fol}
	\varrho(D_{g^{\mathcal{F}}})\otimes (\iota_{ad}^{[0,1)})_!= \varrho(D_{g_\varepsilon^M})\in KK(C,C^*(\mathrm{Mon}(M,TM)_{ad}^{[0,1)}))
    \end{equation}
    where $g_\varepsilon^M$ is induced by $g_{\varepsilon}^{\bar{M}}=g^{\bar{M}/B}\oplus \varepsilon^{-1}g^B$.
\end{theorem}
\begin{proof}
	First, let us  fix the following notation $G:=\mathrm{Mon}(\bar{M},\ker d\pi)$, $G/\Gamma:=\mathrm{Mon}(M,\mathcal{F}(\pi, \Gamma))$,   $H:=\mathrm{Mon}(\bar{M},T\bar{M})$  and $H/\Gamma:=\mathrm{Mon}(M,TM)$. 
Observe that, thanks to Proposition \ref{lift-shriek} and Proposition \ref{lift-rho}  we have that \begin{equation}
\begin{split}
&\varrho(D_{G/\Gamma})\otimes  (\iota_{ad}^{[0,1)})_!=\\  =&\varrho(D_{G/\Gamma})\otimes (\mathcal{M}_G^\Gamma)^{-1}\otimes\mathcal{M}_G^\Gamma\otimes(\iota_{ad}^{[0,1)})_!\otimes(\mathcal{M}_H^\Gamma)^{-1} \otimes\mathcal{M}_H^\Gamma=\\
=& \varrho(D_G^\Gamma)\otimes j^\Gamma((\bar{\iota}_{ad}^{[0,1)})_!)\otimes\mathcal{M}_H^\Gamma=\\
=& \varrho(D_H^\Gamma)\otimes\mathcal{M}_H^\Gamma= \varrho(D_H),
\end{split}
\end{equation}
where the only thing to prove is the equality 
\begin{equation}
\varrho(D_G^\Gamma)\otimes j^\Gamma((\bar{\iota}_{ad}^{[0,1)})_!)= \varrho(D_H^\Gamma)\in KK(\CC, C^*(H_{ad}^{[0,1})\rtimes\Gamma).
\end{equation}
Thanks to \eqref{lift-dirac} and \eqref{lift-shriek}, by using the functoriality of the descent map $j^\Gamma$, it is easy to see that $\varrho(D_G^\Gamma)\otimes j^\Gamma((\bar{\iota}_{ad}^{[0,1)})_!)$ is represented, as in \eqref{mc-class}, by  the pair of bimodules 
\begin{equation}
\left[\left(pt^*\mathcal{E}(H)_{ad}^\rtimes, \psi(D^\Gamma_{H_{ad}})\right),\left(pt^*\mathcal{E}^\rtimes(H)\otimes C[0,1],CS_t(\psi(D^\Gamma_{H}),\mathrm{sgn}(D^\Gamma_G)\otimes 1\right)\right]
\end{equation}
From now the proof follows as from  \eqref{product1} on and we get the desired result.
\end{proof}

\begin{remark}
	Let us consider the more general situation where we have a sequence of $\Gamma$-equivariant surjective Riemannian submersions 
	\begin{equation}
	\xymatrix{\bar{M}\ar[r]^{\pi}& B \ar[r]^q& B' }
 \end{equation}
 and let denote by $\pi'$ the composition $q\circ \pi$.
 Then we have an inclusion of foliations $\ker d\pi/\Gamma\hookrightarrow \ker d\pi'/\Gamma$ over $M$ and then an immersion of Lie groupoids  $$\iota\colon \mathrm{Mon}(M,\mathcal{F}(\pi, \Gamma))\looparrowright \mathrm{Mon}(M,\mathcal{F}(\pi', \Gamma)).$$
Observe that if $B'=pt$ we recover the situation of Theorem \ref{main-th}.
Let $g^\pi$ a longitudinal metric of positive scalar curvature on $\ker d\pi/\Gamma$ and let $g^{\pi'}:=g^\pi\oplus g^{B'}$ be a metric on $\ker d\pi'/\Gamma$ with positive scalar curvature too, then the proof of the following product formula 
\begin{equation}
\varrho(D_{g^{\pi'}})=\varrho(D_{g^{\pi}})\otimes (\iota_{ad}^{[0,1)})_!\in KK_*\left(\CC,C^*(\mathrm{Mon}(M,\mathcal{F}(\pi', \Gamma))_{ad}^{[0,1)})\right)
\end{equation}
follows without substantial changes as the proof of Theorem \ref{main-th}. The only thing to observe is that the estimation of the commutator between the operator of the smaller foliation and the bigger one works exactly as in \cite[Equation (4.46)]{BismutCheeger}.
\end{remark}

\subsection{Functoriality of $\varrho$-classes through \'{e}tale surjections}
In this section we give a very conceptual proof of a generalization to the setting of foliations of \cite[Theorem 1.1]{GXY}. Let $M$ be a smooth compact manifold and let $\mathcal{F}$ be an involutive sub-bundle of $TM$, namely the tangent bundle of a regular foliation on $M$. Consider two $s$-connected Lie groupoids $G$ and $H$ over $M$ integrating the Lie algebroid $\mathcal{F}$ and assume that there exists a Lie groupoid surjective homomorphism $\varphi\colon G\to H$ which integrates the identity map on $\mathcal{F}$. Notice that  $\varphi$ is an \'{e}tale immersion of Lie groupoids, then the normal bundle groupoid of $\varphi$ is $G$ it-self. So we have the following equality of asymptotic morphism classes
\begin{equation}
(\varphi_{ad}^{[0,1)})_!= [\mathrm{ev}_0]^{-1}\otimes [\mathrm{ev}_1]\in E\left(C^*(G_{ad}^{[0,1)}),C^*(H_{ad}^{[0,1)})\right),
\end{equation}
where $\mathrm{ev}_0\colon DNC^{[0,1)}(G_{ad}^{[0,1)},H_{ad}^{[0,1)})\to G_{ad}^{[0,1)}$ and $\mathrm{ev}_1\colon DNC^{[0,1)}(G_{ad}^{[0,1)},H_{ad}^{[0,1)})\to H_{ad}^{[0,1)}$.
Assume that $\mathcal{F}$ is $Spin$  and that it is endowed with a metric $g^{\mathcal{F}}$ with positive scalar curvature. Then we have the following functoriality result. 
\begin{theorem}
	Let $D_G$ and $D_H$ the $Spin$ Dirac operators associated with the metric $g^{\mathcal{F}}$ on $G$ and $H$ respectively, then the following equality holds
	\begin{equation}
	\varrho(D_G)\otimes (\varphi_{ad}^{[0,1)})_!= \varrho(D_H).
	\end{equation}
\end{theorem}
\begin{proof}
	The result follows immediately by noticing that the invertible Dirac operator on the deformation groupoid $DNC(G_{ad}^{[0,1)},H_{ad}^{[0,1)})$, associated with the metric given by the pull-back of $g^{\mathcal{F}}$, restricts to $D_G$ at $0$ and to $D_H$ at $1$.
\end{proof}

\begin{remark}
	When $\mathcal{F}=TM$ the Lie groupoids $G$ and $H$ are of the form $\bar{M}\times_\Lambda\bar{M}\rightrightarrows M$, where $\bar{M}$ is a Galois $\Lambda$-covering of $M$ and we recover \cite[Theorem 1.1]{GXY} when $M$ is compact.
\end{remark}

\section{Open problems}\label{s7}

In this section we are going to list a series of open questions which arise from this paper.
\begin{enumerate}
	\item The first open question comes from the fact that the main result of this paper, namely Theorem \ref{main-th}, is proved for Riemannian foliated bundles. So it is natural to ask for a proof in the context of general foliations which non necessarily admit an isometric normal structure. Notice that in Section \ref{ai} we just give the construction of the adiabatic transverse class for foliations endowed with an almost isometric structure, but the general case could be treated as usual by passing to the Connes' fibration, as it is done in \cite{HSk}. 
	\item Observe now that Definition \ref{rho-def} involves a perturbation which did not appear in the applications of the present work. Indeed, an open problem which will be treated in a future work is to provide a proof of the product formula \eqref{product-fol} for $\varrho$-classes associated with perturbed generalized Dirac operators, such as those given by the signature operator on homotopy equivalent foliations, see \cite[Section 3.4]{zenobi-ad}.
	\item The natural application of these product formulas are given by stability results as in \cite[Section 3.6.1]{zenobi-ad}: namely asking if two longitudinal metric with positive scalar curvature which are not longitudinally concordant stay so if completed to metric with positive scalar curvature on the whole manifold. Similar questions could be asked about foliated homotopy equivalences. In order to prove this kind of results, it is necessary to prove the injectivity of the Kasparov product in \eqref{product-fol}, namely finding suitable hypotheses for the existence of a right inverse (in the notation of Kasparov product) for the lower shriek class between adiabatic deformation groupoids.
	\item Finally, it would be extremely interesting to construct concrete geometrical examples corresponding to non trivial factorizations of $\varrho$-classes as in \eqref{product-fol}. 
\end{enumerate}

\noindent\textbf{Acknowledgments}
I would like to thank Paolo Piazza and Georges Skandalis for interesting discussions about this subject. 

		\addcontentsline{toc}{section}{References}
		\bibliographystyle{plain}
		\nocite{*}
		\bibliography{Product_formulas}

\begin{thebibliography}{10}

\bibitem{An-Sk}
Iakovos Androulidakis and Georges Skandalis.
\newblock A {B}aum-{C}onnes conjecture for singular foliations.
\newblock {\em Ann. K-Theory}, 4(4):561--620, 2019.

\bibitem{AAS}
Paolo Antonini, Sara Azzali, and Georges Skandalis.
\newblock Bivariant {$K$}-theory with {$\Bbb{R}/\Bbb{Z}$}-coefficients and rho
  classes of unitary representations.
\newblock {\em J. Funct. Anal.}, 270(1):447--481, 2016.

\bibitem{BR1}
Moulay-Tahar Benameur and Indrava Roy.
\newblock The {H}igson-{R}oe sequence for \'{e}tale groupoids. {I}. {D}ual
  algebras and compatibility with the {BC} map.
\newblock {\em J. Noncommut. Geom.}, 14(1):25--71, 2020.

\bibitem{BR2}
Moulay-Tahar Benameur and Indrava Roy.
\newblock The {H}igson-{R}oe sequence for \'{e}tale groupoids. {II}. {T}he
  universal sequence for equivariant families.
\newblock {\em J. Noncommut. Geom.}, 15(1):1--39, 2021.

\bibitem{BismutCheeger}
Jean-Michel Bismut and Jeff Cheeger.
\newblock {$\eta$}-invariants and their adiabatic limits.
\newblock {\em J. Amer. Math. Soc.}, 2(1):33--70, 1989.

\bibitem{CSk}
A.~Connes and G.~Skandalis.
\newblock The longitudinal index theorem for foliations.
\newblock {\em Publ. Res. Inst. Math. Sci.}, 20(6):1139--1183, 1984.

\bibitem{Debord-Lescure}
Claire Debord and Jean-Marie Lescure.
\newblock Index theory and groupoids.
\newblock In {\em Geometric and topological methods for quantum field theory},
  pages 86--158. Cambridge Univ. Press, Cambridge, 2010.

\bibitem{EWZ}
Alexander Engel, Christopher Wulff, and Rudolf Zeidler.
\newblock Slant products on the higson-roe exact sequence.
\newblock {\em arXiv:1909.03777}, 2019.

\bibitem{GS-AM}
Alfonso Gracia-Saz and Rajan~Amit Mehta.
\newblock {$\mathcal{VB}$}-groupoids and representation theory of {L}ie
  groupoids.
\newblock {\em J. Symplectic Geom.}, 15(3):741--783, 2017.

\bibitem{GXY}
Hao Guo, Zhizhang Xie, and Guoliang Yu.
\newblock Functoriality for higher rho invariants of elliptic operators.
\newblock {\em J. Funct. Anal.}, 280(10):Paper No. 108966, 36, 2021.

\bibitem{HLS}
N.~Higson, V.~Lafforgue, and G.~Skandalis.
\newblock Counterexamples to the {B}aum-{C}onnes conjecture.
\newblock {\em Geom. Funct. Anal.}, 12(2):330--354, 2002.

\bibitem{Higson-Roe}
Nigel Higson and John Roe.
\newblock Mapping surgery to analysis. {III}. {E}xact sequences.
\newblock {\em $K$-Theory}, 33(4):325--346, 2005.

\bibitem{HSk}
M.~Hilsum and G.~Skandalis.
\newblock Morphismes {$K$}-orient\'{e}s d'espaces de feuilles et
  fonctorialit\'{e} en th\'{e}orie de {K}asparov (d'apr\`es une conjecture
  d'{A}. {C}onnes).
\newblock {\em Ann. Sci. \'{E}cole Norm. Sup. (4)}, 20(3):325--390, 1987.

\bibitem{Hongzhi-Jinmin}
Liu Hongzhi and Wang Jinmin.
\newblock On localized signature and higher rho invariant of fibered manifolds.
\newblock {\em J. Noncommut. Geom.}, 15(3):919--949, 2021.

\bibitem{KaadVS}
Jens Kaad and Walter~D. van Suijlekom.
\newblock Riemannian submersions and factorization of {D}irac operators.
\newblock {\em J. Noncommut. Geom.}, 12(3):1133--1159, 2018.

\bibitem{KvS2}
Jens Kaad and Walter~D. van Suijlekom.
\newblock Factorization of {D}irac operators on almost-regular fibrations of {$
  {\rm spin}^c$} manifolds.
\newblock {\em Doc. Math.}, 25:2049--2084, 2020.

\bibitem{Kasparov81}
G.~G. Kasparov.
\newblock The operator {$K$}-functor and extensions of {$C^{\ast} $}-algebras.
\newblock {\em Izv. Akad. Nauk SSSR Ser. Mat.}, 44(3):571--636, 719, 1980.

\bibitem{Kasparov-equiv}
G.~G. Kasparov.
\newblock Equivariant {$KK$}-theory and the {N}ovikov conjecture.
\newblock {\em Invent. Math.}, 91(1):147--201, 1988.

\bibitem{kasp-sk}
Guennadi Kasparov and Georges Skandalis.
\newblock Groupes ``boliques'' et conjecture de {N}ovikov.
\newblock {\em C. R. Acad. Sci. Paris S\'{e}r. I Math.}, 319(8):815--820, 1994.

\bibitem{MoeMcr}
I.~Moerdijk and J.~Mr\v{c}un.
\newblock {\em Introduction to foliations and {L}ie groupoids}, volume~91 of
  {\em Cambridge Studies in Advanced Mathematics}.
\newblock Cambridge University Press, Cambridge, 2003.

\bibitem{Mohsen}
Omar Mohsen.
\newblock On the deformation groupoid of the inhomogeneous pseudo-differential
  calculus.
\newblock {\em Bull. Lond. Math. Soc.}, 53(2):575--592, 2021.

\bibitem{PS}
Paolo Piazza and Thomas Schick.
\newblock Rho-classes, index theory and {S}tolz' positive scalar curvature
  sequence.
\newblock {\em J. Topol.}, 7(4):965--1004, 2014.

\bibitem{PS2}
Paolo Piazza and Thomas Schick.
\newblock The surgery exact sequence, {K}-theory and the signature operator.
\newblock {\em Ann. K-Theory}, 1(2):109--154, 2016.

\bibitem{Piazza-Zenobi}
Paolo Piazza and Vito~Felice Zenobi.
\newblock Singular spaces, groupoids and metrics of positive scalar curvature.
\newblock {\em J. Geom. Phys.}, 137:87--123, 2019.

\bibitem{Sieg}
Paul Siegel.
\newblock Homological calculations with analytic structure groups.
\newblock {\em PhD thesis}, 2012.

\bibitem{Sk-remarks}
Georges Skandalis.
\newblock Some remarks on {K}asparov theory.
\newblock {\em J. Funct. Anal.}, 56(3):337--347, 1984.

\bibitem{vas}
St{\'e}phane Vassout.
\newblock Unbounded pseudodifferential calculus on {L}ie groupoids.
\newblock {\em J. Funct. Anal.}, 236(1):161--200, 2006.

\bibitem{WXY}
Shmuel Weinberger, Zhizhang Xie, and Guoliang Yu.
\newblock Additivity of higher rho invariants and nonrigidity of topological
  manifolds.
\newblock {\em Comm. Pure Appl. Math.}, 74(1):3--113, 2021.

\bibitem{W}
Christopher Wulff.
\newblock Secondary cup and cap products in coarse geometry.
\newblock {\em Res. Math. Sci.}, 8(3):Paper No. 36, 64, 2021.

\bibitem{XY}
Zhizhang Xie and Guoliang Yu.
\newblock Positive scalar curvature, higher rho invariants and localization
  algebras.
\newblock {\em Adv. Math.}, 262:823--866, 2014.

\bibitem{XY2}
Zhizhang Xie and Guoliang Yu.
\newblock Higher rho invariants and the moduli space of positive scalar
  curvature metrics.
\newblock {\em Adv. Math.}, 307:1046--1069, 2017.

\bibitem{Zeidler}
Rudolf Zeidler.
\newblock Positive scalar curvature and product formulas for secondary index
  invariants.
\newblock {\em J. Topol.}, 9(3):687--724, 2016.

\bibitem{Z}
Vito~Felice Zenobi.
\newblock Mapping the surgery exact sequence for topological manifolds to
  analysis.
\newblock {\em J. Topol. Anal.}, 9(2):329--361, 2017.

\bibitem{zenobi-ad}
Vito~Felice Zenobi.
\newblock Adiabatic groupoid and secondary invariants in {K}-theory.
\newblock {\em Adv. Math.}, 347:940--1001, 2019.

\bibitem{Z-CvsAd}
Vito~Felice Zenobi.
\newblock The adiabatic groupoid and the {H}igson--{R}oe exact sequence.
\newblock {\em J. Noncommut. Geom.}, 15(3):797--827, 2021.

\end{thebibliography}
		
\end{document}